\documentclass[a4paper]{imanum}
\pdfoutput=1
\newif\ifarxiv
\arxivtrue

\ifarxiv
\cropmark{N}
\gridframe{N}
\makeatletter
\let\ps@plain\relax
\def\ps@plain{\def\@oddhead{}}
\makeatother
\else
\fi

\usepackage[percent]{overpic}
\usepackage[T1]{fontenc}
\usepackage[utf8]{inputenc}
\usepackage{amsmath}
\usepackage{amscd}
\usepackage{amssymb}
\usepackage{graphicx}
\usepackage{url}
\usepackage{color}
\usepackage{listings}
\usepackage{booktabs}
\usepackage{multirow}
\usepackage{textcomp}
\usepackage{mathtools}
\usepackage{tikz}
\usepackage{expl3}
\usetikzlibrary{trees,calc, positioning,decorations.markings,arrows,backgrounds}
\usepackage{pgfplots}
\usetikzlibrary{plotmarks}
\usepackage{pgfplotstable}
\usepackage{environ}
\usepackage{algorithm}
\usepackage{enumerate}
\usepackage[noend]{algpseudocode}
\usepackage{color}
\usepackage{xspace}
\usepackage{svg}
\usepackage{standalone}

\let\div\relax
\DeclareMathOperator{\div}{div}

\DeclareMathOperator{\grad}{grad}

\newcommand{\mesh}{{\cal T}}

\newcommand{\tbnorm}[1]{\vert\kern-0.1em\vert\kern-0.1em\vert\, #1 \,\vert\kern-0.1em\vert\kern-0.1em\vert}

\newcommand{\qbnorm}[1]{\vert\kern-0.15em\vert\kern-0.15em\vert\kern-0.15em\vert\, #1 \,\vert\kern-0.15em\vert\kern-0.15em\vert\kern-0.15em\vert}

\definecolor{blue}{rgb}{0.2980392156862745, 0.4470588235294118, 0.6901960784313725}
\definecolor{green}{rgb}{0.3333333333333333, 0.6588235294117647, 0.40784313725490196}
\definecolor{red}{rgb}{0.7686274509803922, 0.3058823529411765, 0.3215686274509804}

\newcommand{\hg}{\ensuremath{{h,\gamma}}}
\newcommand{\Hg}{\ensuremath{{H,\gamma}}}

\newcommand{\Ptwo}{\ensuremath{\mathbb{P}_2}\xspace}
\newcommand{\Pone}{\ensuremath{\mathbb{P}_1}\xspace}

\newcommand{\PkdPminusdisc}{\ensuremath{[\mathbb{P}_k]^d\mathrm{-}\mathbb{P}_{k-1}^\text{disc}}\xspace}
\newcommand{\Pthree}{\ensuremath{\mathbb{P}_3}\xspace}
\newcommand{\PtwoPzero}{\ensuremath{[\mathbb{P}_2]^2\mathrm{-}\mathbb{P}_0}\xspace}

\DeclareMathOperator{\MacroStar}{\ensuremath{\mathrm{macro}\,\mathrm{star}}}

\usepackage{bm}
\usepackage{xparse}
\usepackage{relsize}
\newcommand{\re}{\mathbb{R}}

\newcommand{\sgrad}[1]{\ensuremath{\mathrm{E}{#1}}}

\newcommand{\Nc}{\mathcal{N}}

\newcommand{\fb}{\mathbf{f}}

\newcommand{\hb}{\mathbf{h}}
\newcommand{\nb}{\mathbf{n}}

\newcommand{\ub}{\mathbf{u}}
\newcommand{\vb}{\mathbf{v}}

\newcommand{\zerob}{\mathbf{0}}

\newcommand{\Ib}{\mathbf{I}}

\let\temp\phi
\let\phi\varphi
\let\varphi\temp

\newcommand{\Phib}{\mathbf{\Phi}}

\newcommand{\phib}{\boldsymbol{\phi}}

\newcommand{\vertiii}[1]{{\left\vert\kern-0.25ex\left\vert\kern-0.25ex\left\vert #1 
    \right\vert\kern-0.25ex\right\vert\kern-0.25ex\right\vert}}

\def\Proj_#1{\ensuremath{\operatorname{Proj}_{#1}\!}}

\newcommand{\dr}[1]{\ensuremath{\operatorname{d}\!{#1}}}

\def\drs_#1{\ensuremath{\operatorname{d}_{#1}\!}}
\def\Drs_#1{\ensuremath{\operatorname{D}_{#1}\!}}
\def\Drs^#1{\ensuremath{\operatorname{D}^{#1}\!}}

\def\Ers_#1{\mathrm{E}_{#1}}

\DeclareMathOperator{\supp}{supp}
\let\div\relax
\DeclareMathOperator{\div}{div}
\DeclareMathOperator{\curl}{curl}

\newcommand{\defeq}{\vcentcolon=}

\def\xunderbrace#1_#2{{\underbrace{#1}_{\mathclap{#2}}}}
\def\xoverbrace#1^#2{{\overbrace{#1}^{\mathclap{#2}}}}
\def\xunderarrow#1_#2{{\underset{\overset{\uparrow}{\mathclap{#2}}}{#1}}}
\def\xoverarrow#1^#2{{\overset{\underset{\downarrow}{\mathclap{#2}}}{#1}}}

\DeclareDocumentCommand{\boxedeq}{m o}{%
    \IfNoValueTF{#2}{%
        \rlap{\boxed{#1}}%
        \phantom{\hskip\fboxrule\hskip\fboxsep #1}%
        }{%
        \rlap{\boxed{#1#2}}%
        \phantom{\hskip\fboxrule\hskip\fboxsep #1}&\phantom{#2}%
    }%
}

\NewDocumentCommand\meq{m+g}{%
\IfNoValueTF{#2}
  {\overset{\mathsmaller{#1}}{=}}
  {\overset{\mathsmaller{#1}}{\underset{\mathsmaller{#2}}=}}%
}
\catcode`@=11

\def\nvphantom{\v@true\h@false\nph@nt}
\def\nhphantom{\v@false\h@true\nph@nt}
\def\nphantom{\v@true\h@true\nph@nt}
\def\nph@nt{\ifmmode\def\next{\mathpalette\nmathph@nt}%
\else\let\next\nmakeph@nt\fi\next}
\def\nmakeph@nt#1{\setbox\z@\hbox{#1}\nfinph@nt}
\def\nmathph@nt#1#2{\setbox\z@\hbox{$\m@th#1{#2}$}\nfinph@nt}
\def\nfinph@nt{\setbox\tw@\null
    \ifv@ \ht\tw@\ht\z@ \dp\tw@\dp\z@\fi
\ifh@ \wd\tw@-\wd\z@\fi \box\tw@}

\def\XXint#1#2#3{{\setbox0=\hbox{$#1{#2#3}{\int}$ }
\vcenter{\hbox{$#2#3$ }}\kern-.6\wd0}}

\makeatletter
\newcommand{\pushright}[1]{\ifmeasuring@#1\else\omit\hfill$\displaystyle#1$\fi\ignorespaces}
\newcommand{\pushleft}[1]{\ifmeasuring@#1\else\omit$\displaystyle#1$\hfill\fi\ignorespaces}
\makeatother
\newcommand{\review}[1]{#1}

\definecolor{darkblue}{rgb}{0.00,0.00,0.55}
\definecolor{black}{rgb}{0.00,0.00,0.00}

\begin{document}

\title{Robust multigrid methods for nearly incompressible elasticity using macro elements}
\shorttitle{Robust multigrid via Fortin operators}

\author{%
{\sc
Patrick~E.~Farrell\,\thanks{Email: patrick.farrell@maths.ox.ac.uk}}\\[2pt]
Mathematical Institute, University of Oxford, Oxford, UK\\[6pt]
{\sc and}\\[6pt]
{\sc
Lawrence Mitchell\,\thanks{Email: lawrence.mitchell@durham.ac.uk}}\\[2pt]
Department of Computer Science, Durham University, Durham, UK\\[6pt]
{\sc and}\\[6pt]
{\sc
L.~Ridgway Scott\,\thanks{Email: ridg@uchicago.edu}}\\[2pt]
Department of Computer Science, University of Chicago, Chicago, USA\\[6pt]
{\sc and}\\[6pt]
{\sc
Florian~Wechsung\,\thanks{Corresponding author. Email: wechsung@nyu.edu}}\\[2pt]
Courant Institute of Mathematical Sciences, New York University, New York, USA
}

\shortauthorlist{P.~E.~Farrell, L.~Mitchell, L.~R.~Scott, and F.~Wechsung}

\maketitle

\begin{abstract}
{%
We present a mesh-independent and parameter-robust multigrid solver for
the Scott--Vogelius discretisation of the nearly incompressible linear
elasticity equations on meshes with a macro element structure. The
discretisation achieves exact representation of the limiting divergence
constraint at moderate polynomial degree. Both the relaxation and
multigrid transfer operators exploit the macro structure for robustness
and efficiency. For the relaxation, we use the existence of local Fortin
operators on each macro cell to construct a local space decomposition
with parameter-robust convergence. For the transfer, we construct a robust
prolongation operator by performing small local
solves over each coarse macro cell. The necessity of both components
of the algorithm is confirmed by numerical experiments.
}{linear elasticity; multigrid; preconditioning; macro elements; parameter-robustness}
\end{abstract}

\section{Introduction}

We consider the linear elasticity equations on a simply connected domain $\Omega\subset\re^d$ with boundary $\partial\Omega = \Gamma_N\cup\Gamma_D$, given by 
\begin{equation} \label{eqn:elasticity}
    \begin{aligned}
        -\nabla\cdot( \sgrad\ub + \gamma (\nabla \cdot \ub) \Ib_d  ) &= \fb && \text{ in } \Omega,\\
            \ub &= \zerob && \text{ on } \Gamma_D,\\
            (\sgrad \ub + \gamma (\nabla\cdot \ub) \Ib_d)\nb &= \hb && \text{ on } \Gamma_N,
    \end{aligned}
\end{equation}
where $\sgrad \ub = \frac12(\nabla \ub + \nabla \ub^\top)$, $\gamma \ge 0$,  and $\fb$ and $\hb$ are given data.
Here $\gamma={\lambda}/{2\mu}$ where $\mu$ and $\lambda$ are the Lam\'e parameters describing an isotropic, homogeneous material.
As $\gamma\to\infty$ this corresponds to the nearly incompressible case and the equations become difficult to solve~\citep{schoberl1999b,lee2009, dohrmann_overlapping_2009}.

In weak form the elasticity equations can be expressed as: given $\fb \in H^{-1}(\Omega;\re^d)$ and $\hb\in H^{-1/2}(\Gamma_N;\re^d)$, find $\ub\in V := \{ \vb \in H^1(\Omega;\re^d) : \vb\vert_{\Gamma_D}=\zerob\}$ such that
\begin{equation}
    (\sgrad{\ub}, \sgrad \vb) + \gamma (\nabla\cdot \ub, \nabla\cdot \vb) = \langle \fb, \vb\rangle  + \langle \hb, \vb\rangle 
\end{equation}
for all $\vb \in V$.
Here $(\cdot, \cdot)$ denotes the standard $L^2$ inner product and $\langle\cdot, \cdot \rangle$ refers to the dual pairing.
The numerical solution of these equations in the nearly incompressible limit using the finite element method has attracted significant attention.
Choosing finite dimensional subspaces $V_h\subset V$ and $Q_h\subseteq \div(V_h) \subset L^2(\Omega)$, we consider the problem: find $\ub\in V_h$ such that
\begin{equation}
    a_\hg(\ub, \vb) = \langle \fb, \vb\rangle  + \langle \hb,
    \vb\rangle \qquad \text{ for all } \vb \in V_h.
\end{equation}
Here $a_\hg$ is defined as
\begin{equation}\label{eqn:paramdependent-bil-form}
    a_\hg(\ub, \vb) \defeq a(\ub, \vb) + \gamma c_h(\ub, \vb),
\end{equation}
and 
\begin{equation}
    \begin{aligned}
        a(\ub, \vb) &\defeq (\sgrad{\ub}, \sgrad \vb),\\
        c_h(\ub, \vb)&\defeq (\Pi_{Q_h}(\nabla\cdot \ub), \Pi_{Q_h}(\nabla\cdot \vb)),
    \end{aligned}
\end{equation}
where $\Pi_{Q_h}$ is the projection onto $Q_h$.
Here $Q_h$ is the pressure space in an associated mixed formulation of the problem.
We note that these equations also arise in iterated penalty and augmented Lagrangian formulations of the Stokes and Navier--Stokes equations.
In this context it is highly desirable that the property $\div(H^1) = L^2$ is preserved for the discretised problem, that is $\div(V_h) = Q_h$~\citep{john2017}.
An example of a such a discretisation is the \PkdPminusdisc Scott--Vogelius discretisation of degree $k$.
In this work we restrict ourselves to discretisations preserving this exactness property.

As $\gamma$ is increased, two separate issues arise: first, locking may occur~\citep{ManIvoLockingElasticity}.
This can be avoided by considering discretisations $V_h\times Q_h$ that satisfy a discrete inf-sup condition.
The second issue is that the arising linear systems become poorly conditioned since the problem becomes nearly singular, due to the nullspace of the divergence operator.
The focus of this work is the development of a multigrid method with robust performance as $\gamma$ becomes large for the Scott--Vogelius discretisation.

Sch\"oberl~developed a robust multigrid scheme for nearly incompressible elasticity in his doctoral thesis \citep{schoberl1999,schoberl1999b}.
He proved that the key ingredients are a robust prolongation and a robust relaxation scheme and showed how to construct these for the $\PtwoPzero$ element in two dimensions.
His insight was that the prolongation needs to map divergence-free vector fields on the coarse grid to (nearly) divergence-free vector fields on the fine grid and that a robust relaxation can be built from a space decomposition that provides a stable local decomposition of the kernel of the divergence.
While Sch\"oberl considered additive relaxation, \citet{lee2007, lee2009} later developed a similar kernel decomposition condition for parameter-robust multiplicative relaxation.

The proof by Sch\"oberl for the kernel decomposition is based on an explicit construction of a Fortin operator that is used in the proof of inf-sup stability for the $\PtwoPzero$ element.
Lee et al.~consider the Scott--Vogelius element, but rely on the existence of a local basis for $C^1$ piecewise polynomials to construct this decomposition.
As the existence of such a space is only known for high-order polynomials, they have to consider at least $k \ge 4$ and are limited to two dimensions~\citep{morgan_nodal_1975}.
Also relying on a local basis for $C^1$ piecewise polynomials,~\citet{wu2014} extend the results of Lee et al.~to the additive case.
Lastly we mention that it is also possible to develop domain decomposition preconditioners for this problem, as shown by~\citet{dohrmann_overlapping_2009}.

\subsection{Contributions}
In this work we consider meshes with a macro element structure with the main requirement being that the inf-sup condition is satisfied on each macro element.
Our main contribution is the construction of a particular localised Fortin operator obtained by ``gluing'' together Fortin operators on each macro element, the existence of which is guaranteed by inf-sup stability on the macro elements.
Using this operator we are then able to construct robust relaxation and prolongation schemes for lower polynomial degrees $k$ than previously possible.
As examples we consider the Scott--Vogelius element on Alfeld and Powell--Sabin splits. 
We emphasize that in contrast to previous work considering the Scott--Vogelius element in this context, our approach does not require an explicit construction of a local basis of a $C^1$ space.
Finally, we note that the strategy developed here can be applied to find robust multigrid preconditioners for other elements.
These include the popular $[\mathbb{Q}_k]^d\mathrm{-}\mathbb{P}_{k-1}^\text{disc}$ or the $[\mathbb{Q}_k]^d\mathrm{-}\mathbb{Q}_{k-2}^\text{disc}$ elements~\citep{bernardi_spectral_1997,schwab_mixed_1999,matthies_inf-sup_2002,heuveline_inf-sup_2007} on quadrilateral and hexahedral meshes, since for these elements inf-sup stability is usually proven using the macro element approach.


\subsection{Structure}
The rest of this work is structured as follows.
In Section~\ref{sec:smoothing} we develop robust smoothers in the context of subspace correction methods.
In Section~\ref{sec:prolongation} we motivate the need for a special
prolongation operator and give a different proof of robustness for a modification of the prolongation introduced previously by Sch\"oberl.
In Section~\ref{sec:mgrid} we state a convergence theorem for multigrid W-cycles using these two ingredients.
Finally, in Section~\ref{sec:numerics} we report numerical examples in two and in three dimensions that clearly demonstrate that standard geometric and algebraic multigrid methods fail for this problem, and that both the robust smoothing and robust prolongation are necessary to obtain an effective solver in the high $\gamma$ limit.



\section{Smoothing}\label{sec:smoothing}

We define the operator $A_\hg : V_h^{\phantom{*}} \to V_h^*$ by 
\begin{equation}
    \langle A_\hg \ub , \vb\rangle \defeq a_\hg(\ub, \vb),
\end{equation}
and drop the subscript $\gamma$ to denote the case of $\gamma=0$, i.e.
\begin{equation}
    \langle A_h \ub , \vb\rangle \defeq a(\ub, \vb).
\end{equation}

Many smoothers commonly used in multigrid can be expressed as
subspace correction methods \citep{xu1992}.
We consider a decomposition
\begin{equation}
    V_h = \sum_i V_i
\end{equation}
where the sum is not necessarily direct.
For each subspace $i$ we denote the natural inclusion by $I_i:V_i \to V_h$
and we define the restriction $A_i$ of $A_\hg$ onto $V_i$ as 
\begin{equation}
    \langle A_i \ub_i, \vb_i\rangle \defeq \langle A_\hg I_i \ub_i, I_i \vb_i\rangle \qquad \text{ for all } \ub_i,\vb_i\in V_i.
\end{equation}
The additive Schwarz preconditioner associated with the space decomposition $\{V_i\}$ is then defined by the action of its inverse:
\begin{equation}
    D_\hg^{-1} := \sum_i I_i A_i^{-1} I_i^*.
\end{equation}
\review{In this work we focus on the case when exact solves are used on each subspace, but we note that the work of \citet{xu1992} also considers inexact solves, i.e.~$D_\hg^{-1} := \sum_i I_i R_i I_i^*$ with $R_i\approx A_i^{-1}$.}

The method is also known as the parallel subspace correction method~\citep{xu1992}.
We denote the norms induced by the operator and the preconditioner by
\begin{equation}
    \begin{aligned}
        \|\ub\|_{A_\hg}^2 &= \langle A_\hg\ub, \ub\rangle\\
        \|\ub\|_{D_\hg}^2 &= \langle D_\hg\ub, \ub\rangle
    \end{aligned}
\end{equation}
and recall a key result in the theory of subspace correction methods
(\citet[Lemma 1]{widlund1992}, see \citet[Theorem 4.1]{schoberl1999b}
or \citet[Eqn.~(4.11)]{xu_method_2001} for a proof):
\begin{equation}\label{eqn:ssc-norm-expression}
    \|\ub_h\|_{D_\hg}^{2}  = \inf_{\substack{\ub_i \in V_i\\\sum_i \ub_i = \ub_h}} \sum_i \|\ub_i\|^2_{A_i}.
\end{equation}
To prove spectral equivalence of $D_\hg$ and $A_\hg$, we want to obtain a bound of the form
\begin{equation}\label{eqn:spectral-bounds-general}
    c_1 D_\hg \le A_\hg \le c_2 D_\hg,
\end{equation}
where $M\le N$ means that $\|\ub\|_M \le \|\ub\|_N$ for all $\ub$.
A bound for the number of iterations required by the conjugate gradient method for $A_\hg$ preconditioned by $D_\hg$ then behaves like $\sqrt{c_2/c_1}$~\cite[eqn.~(2.18)]{Elman:2014aa}.

The second inequality in~\eqref{eqn:spectral-bounds-general} measures
the interaction between subspaces and can be bounded in a parameter-independent way by the maximum overlap of the subspaces $N_O$~(\citet[Lemma 4.6]{xu1992}, \citet[Lemma 3.2]{schoberl1999b}),
which is bounded on shape regular meshes.

The first inequality in~\eqref{eqn:spectral-bounds-general} is harder to obtain and usually depends not only on the smoother but also on the PDE and the mesh size.
%
%
We demonstrate this for the case of Jacobi relaxation, i.e.~when $V_i = \{\alpha\phib_i: \alpha\in \mathbb{R}\}$ where $\{\phib_i\}$ is the set of basis functions used for $V_h$.
We note that the decomposition $\ub_h=\sum \ub_i$, $\ub_i\in V_i$ is unique, and hence,
\begin{equation} \label{eqn:jacobi-failure}
    \begin{aligned}
        \|\ub_h\|_{D_\hg}^2 &= \sum_i \|\ub_i\|_{A_\hg}^2 \preceq (1+\gamma) \sum_i \|\ub_i\|_1^2 \preceq \frac{1+\gamma}{h^2} \sum_i \|\ub_i\|_0^2 \\
                          &\preceq (1+\gamma) h^{-2} \|\ub_h\|_0^2 \preceq (1+\gamma) h^{-2} \|\ub_h\|_{A_\hg}^2,
    \end{aligned}
\end{equation}
\review{where $a\preceq b$ means that there exists a constant $C$ independent of $h$ and $\gamma$ such that $a \le C b$.}
This bound is parameter dependent and degrades for large $\gamma$, i.e.~$D_\hg$ becomes a poor preconditioner for $A_\hg$ as $\gamma\to\infty$.

To obtain a bound independent of $\gamma$, one requires a space decomposition that respects the nullspace of the divergence operator, which we denote by
\begin{equation}
    \Nc_h = \{ \vb_h \in V_h : \Pi_{Q_h} (\nabla\cdot \vb_h) = 0 \}.
\end{equation}
To build intuition, we consider a $\ub_0\in \Nc_h$.
If the space decomposition satisfies
\begin{equation}\label{eqn:kernel-decomposition-property}
    \mathcal{N}_h = \sum_i V_i \cap \mathcal{N}_h,
\end{equation}
then $\ub_0$ can be written as
\begin{equation}\label{eqn:splitting-exists}
    \ub_0 = \sum_i \ub_{0,i}, \qquad \ub_{0,i}\in V_i\cap\mathcal{N}_h.
\end{equation}
Redoing the first steps of the calculation
in~\eqref{eqn:jacobi-failure}, and using that each of the $\ub_{0, i}$
are divergence-free, the second term
in~\eqref{eqn:paramdependent-bil-form} vanishes and we obtain
\begin{equation}\label{eqn:u0-ssc}
    \|\ub_0\|_{D_\hg}^2 \le \sum_i \|\ub_{0, i}\|_{A_\hg}^2 \preceq \sum_i \|\ub_{0, i}\|_1^2.
\end{equation}
We now make this idea rigorous and prove $\gamma$-independent spectral equivalence of $D_\hg$ and $A_\hg$.
A key assumption, which we will need to check for each element and space decomposition individually, is that the splitting in~\eqref{eqn:splitting-exists} is stable, so that the last term in~\eqref{eqn:u0-ssc} can be bounded.
This statement was proven for general parameter-dependent problems by Sch\"oberl but for completeness we include a proof for the special case of elasticity.

\begin{proposition}[{{\citet[Theorem 4.1]{schoberl1999b}}}]\label{prop:kernel-decomposition-theorem}
    Let $\{V_i\}$ be a space decomposition of $V_h$ with overlap $N_O$ and assume that the pair $V_h\times Q_h$ is inf-sup stable for the mixed problem
    \begin{equation}
        B((\ub, p), (\vb, q)) \defeq a(\ub, \vb) - (\nabla\cdot \vb, p) - (\nabla\cdot \ub, q).
    \end{equation}
    Assume that any $\ub_h\in V_h$ and $\ub_0 \in \mathcal{N}_h$ satisfy
    \begin{equation} \label{eqn:smoothing-bounds-requirements}
        \begin{aligned}
            \inf_{\substack{\ub_h=\sum \ub_i\\\ub_i\in V_i}} \sum_{i} \|\ub_i\|_{1}^2 &\le c_1(h) \|\ub_h\|_{0}^2,\\
            \inf_{\substack{\ub_0=\sum \ub_{0, i}\\\ub_{0,i}\in
    \mathcal{N}_h \cap V_i}} \sum_{i} \|\ub_{0,i}\|_{1}^2 &\le c_2(h) \|\ub_0\|_{0}^2,
        \end{aligned}
    \end{equation}
    \review{where $\|\cdot\|_0$ denotes the standard $L^2$ norm, and $\|\cdot\|_1$ denotes the standard $H^1$ norm.}
    Then it holds that
    \begin{equation}
        (c_1(h) + c_2(h))^{-1} D_\hg \preceq A_\hg \le N_O D_\hg,
    \end{equation}
    with constants independent of $\gamma$.
\end{proposition}
\begin{proof}
Let $\ub_h\in V_h$, and consider a decomposition $\ub_h = \ub_0 + \ub_1$ obtained by solving
\begin{equation}\label{eqn:decomposition-u1-property}
    B((\ub_1, p_1), (\vb_h, q_h)) = (\nabla\cdot \ub_h, q_h) \qquad \text{for all } (\vb_h, q_h)\in V_h\times Q_h.
\end{equation}
Testing with $\vb_h=0$ we obtain that $\Pi_{Q_h} (\nabla\cdot \ub_1) = \Pi_{Q_h} (\nabla\cdot \ub_h)$ and hence $\Pi_{Q_h} (\nabla \cdot \ub_0) = 0$.
Furthermore, \review{denoting $\|(\vb_h, q_h)\| = \|\vb_h\|_1 + \|q_h\|_0$}, by stability we have 
\begin{equation}\label{eqn:bounds-for-u1-by-stability}
    \begin{aligned}
        \|\ub_1\|_{1} & \preceq \sup_{\substack{\vb_h\in V_h\\ q_h\in Q_h}} \frac{B((\ub_1, p_1), (\vb_h, q_h))}{\|(\vb_h, q_h)\|} \\
                    & \overset{\mathclap{\eqref{eqn:decomposition-u1-property}}}\le\sup_{\substack{\vb_h\in  V_h\\  q_h\in Q_h}} \frac{\|\Pi_{Q_h} (\nabla\cdot \ub_h)\|_0 \|q_h\|_0}{\|(\vb_h, q_h)\|}\\
                    & \le\|\Pi_{Q_h} (\nabla\cdot \ub_h)\|_0
    \end{aligned}
\end{equation}
and hence $\|\ub_1\|_1 \preceq \|\ub_h\|_1$ and $\|\ub_1\|_1\preceq \gamma^{-1/2} \|\ub_h\|_{A_\hg}$.
Using $\ub_0 = \ub_h - \ub_1$ we obtain in addition that $\|\ub_0\|_1\preceq \|\ub_h\|_1$ and conclude
\begin{equation}
    \begin{aligned}
        \|\ub_h\|_{D_h}^2 & \le \inf_{\substack{\ub_1=\sum \ub_{1, i}\\\ub_{1,i}\in V_i}} \sum_{i} \xunderbrace{\|\ub_{1,i}\|_{A_\hg}^2}_{\le (1+\gamma)\|\ub_{1,i}\|_1^2} +  \inf_{\substack{\ub_0=\sum \ub_{0, i}\\\ub_{0,i}\in \mathcal{N}_h \cap V_i}} \sum_{i} \xunderbrace{\|\ub_{0,i}\|_{A_\hg}^2}_{=\|\ub_{0,i}\|_1^2} \\[2pt]
                        &\overset{\mathclap{\eqref{eqn:smoothing-bounds-requirements}}}\preceq (1+\gamma) c_1(h) \|\ub_1\|_0^2 + c_2(h) \|\ub_0\|_0^2\\
                        &\preceq (1+\gamma) c_1(h) \|\ub_1\|_1^2 + c_2(h) \|\ub_0\|_1^2\\
                        &\preceq (c_1(h) + c_2(h)) \|\ub_h\|_{A_\hg}^2.
    \end{aligned}
\end{equation}
\end{proof}

We now discuss two approaches to finding space decompositions $\{V_i\}$ that satisfy the kernel decomposition property in~\eqref{eqn:kernel-decomposition-property}.
\subsection{Characterisation of the kernel using exact de Rham complexes}
We begin by recalling some fundamental de Rham complexes.

The smooth de Rham complex in two dimensions is given by
\begin{equation}
  \label{eqn:smooth-complex-2d}
  \mathbb{R} \xrightarrow{\operatorname{id}} C^\infty(\Omega)
  \xrightarrow{\curl} [C^\infty(\Omega)]^2 \xrightarrow{\div} C^\infty(\Omega) \xrightarrow{\operatorname{null}} 0,
\end{equation}
and in three dimensions
\begin{equation}
  \label{eqn:smooth-complex-3d}
  \mathbb{R} \xrightarrow{\operatorname{id}} C^\infty(\Omega)
  \xrightarrow{\grad} [C^\infty(\Omega)]^3 \xrightarrow{\curl} [C^\infty(\Omega)]^3 \xrightarrow{\div} C^\infty(\Omega) \xrightarrow{\operatorname{null}} 0.
\end{equation}
Such a complex is called exact if the kernel of an operator is given by the range of the preceding operator in the sequence, e.g.~when $\mathrm{range}\curl = \mathrm{ker}\div$.
It is well known that these complexes are exact precisely when the domain is simply connected \cite[p.~18]{arnold_finite_2006}.
Such an exactness property is of interest here because it allows us to characterise divergence-free vector fields as the curls of potentials.

Several lower regularity variants of these complexes exist. Likely the best-known ones are the complexes
\begin{align}
  \label{eqn:lowreg-complex-2d}
    &\qquad\quad \mathbb{R} \xrightarrow{\operatorname{id}} H^1(\Omega)
    \xrightarrow{\curl} H(\div, \Omega) \xrightarrow{\div} L^2(\Omega) \xrightarrow{\operatorname{null}} 0,&\pushright{\text{(2D)}\hspace{-0.9cm}}\\
  \label{eqn:lowreg-complex-3d}
    &\mathbb{R} \xrightarrow{\operatorname{id}} H^1(\Omega)
    \xrightarrow{\grad} H(\curl, \Omega) \xrightarrow{\curl} H(\div, \Omega) \xrightarrow{\div} L^2(\Omega) \xrightarrow{\operatorname{null}} 0.&\pushright{\text{(3D)}\hspace{-0.9cm}}
\end{align}
In the last decades, a significant effort has been made to find finite element spaces that form exact subcomplexes of~\eqref{eqn:lowreg-complex-2d} and~\eqref{eqn:lowreg-complex-3d}~\citep{arnold_finite_2006}.
For this work, we are interested in characterising the kernel of the divergence of vector fields with $H^1$ regularity.
Hence, we study the so-called Stokes complexes where the function spaces enjoy higher regularity, given by
\begin{equation}
  \label{eq:stokes-complex-2d}
  \mathbb{R} \xrightarrow{\operatorname{id}} H^2(\Omega)
  \xrightarrow{\curl} [H^1(\Omega)]^2 \xrightarrow{\div} L^2(\Omega) \xrightarrow{\operatorname{null}} 0,
\end{equation}
and in three dimensions
\begin{equation}
  \label{eq:stokes-complex-3d}
  \mathbb{R} \xrightarrow{\operatorname{id}} H^2(\Omega)
  \xrightarrow{\grad} H^1(\curl, \Omega) \xrightarrow{\curl} [H^1(\Omega)]^3 \xrightarrow{\div} L^2(\Omega) \xrightarrow{\operatorname{null}} 0,
\end{equation}
where $H^1(\curl, \Omega) = \{\ub \in [H^1(\Omega)]^3: \curl \ub \in [H^1(\Omega)]^3 \}$.
Discrete subcomplexes of these Stokes complexes are much harder to construct and often result in high order polynomials due to the high regularity requirements.

Assume now that we have been given a discrete exact subsequence of~\eqref{eq:stokes-complex-2d} or~\eqref{eq:stokes-complex-3d} of the form
\begin{equation}
 \cdots \rightarrow \Sigma_h \xrightarrow{\curl} V_h \xrightarrow{\div} Q_h \xrightarrow{\operatorname{null}} 0.
\end{equation}
Then for a divergence-free discrete vector field $\ub_h\in V_h$, we can write it as the $\curl$ of a potential $\Phib_h\in\Sigma_h$.
We note that $\Phib_h$ is a vector field in three dimensions but a scalar field in two dimensions.

Assume $\Sigma_h$ has a basis given by $\{\Phib_j\}$, so that $\Phib_h$ can be written as $\Phib_h = \sum_j c_j \Phib_j$ for some coefficients $c_j$.
Now we can define a divergence-free decomposition of $\ub_h$ as $\ub_h = \sum_j \ub_j$ where $\ub_j = c_j \nabla\times \Phib_j$.
Hence, a space decomposition $\{V_i\}$ such that $\nabla \times \Phib_j \in V_i$ for some $i$ for all basis functions $\Phib_j$ decomposes the kernel.
To understand how to choose a decomposition $\{V_i\}$ that satisfies this property, we have to examine the support of the basis functions $\{\Phib_j\}$.

In two dimensions and on barycentrically refined meshes, choosing $\Sigma_h$ to be the Hsieh-Clough-Tocher (HCT) finite element space together with continuous $[\Ptwo]^2$ finite element functions for $V_h$ and discontinuous $\Pone$ finite element functions for $Q_h$ yields an exact discrete complex~\cite[p.~514]{john2017}.
The three elements are displayed in Figure~\ref{fig:hct-complex}.
\begin{figure}
    \begin{center}
        \begin{tikzpicture}[scale=4,
    normaldofs/.style={postaction=decorate, decoration={markings, mark=at position 0.5 with {
            \draw (0,-5pt) -- ++(0, 10pt);
    }}},
    pointdofs/.style={postaction=decorate, decoration={markings, mark=at position 0.5 with {
                \draw [fill] circle (0.12);
    }}}
    ]
    \coordinate (0) at (0, 0) {};
    \coordinate (1) at (1, 0) {};
    \coordinate (2) at ($(0) + (60:1)$) {};
    \coordinate (bc) at (barycentric cs:0=1,1=1,2=1) {};

    \draw[normaldofs] (0) to (1);
    \draw[normaldofs] (1) to (2);
    \draw[normaldofs] (2) to (0);

    \foreach \i in {0, 1, 2} {
        \draw (\i) circle (0.06);
        \draw[fill] (\i) circle (0.03);
    }

    \foreach \i in {0, 1, 2} {
        \draw[dashed] (bc) -- (\i);
    }
    \begin{scope}[shift={(1.5, 0)}]

        \coordinate (0) at (0, 0) {};
        \coordinate (1) at (1, 0) {};
        \coordinate (2) at ($(0) + (60:1)$) {};
        \coordinate (bc) at (barycentric cs:0=1,1=1,2=1) {};

        \draw[pointdofs] (0) to (1);
        \draw[pointdofs] (1) to (2);
        \draw[pointdofs] (2) to (0);
        \foreach \i in {0, 1, 2} {
            \draw[pointdofs] (bc) -- (\i);
        }

        \foreach \i in {0, 1, 2} {
            \draw[fill] (\i) circle (0.03);
        }

        \draw[fill] (bc) circle (0.03);
    \end{scope}
    \begin{scope}[shift={(3, 0)}]
        \coordinate (0) at (0, 0) {};
        \coordinate (1) at (1, 0) {};
        \coordinate (2) at ($(0) + (60:1)$) {};
        \coordinate (bc) at (barycentric cs:0=1,1=1,2=1) {};
        \draw[line join=miter] (0) -- (1) -- (2) -- cycle;
        \draw[line join=miter] (0) -- (bc) -- (1);
        \draw[line join=miter] (2) -- (bc);

        \coordinate (bc0) at (barycentric cs:0=1,1=1,bc=1) {};
        \coordinate (bc1) at (barycentric cs:1=1,2=1,bc=1) {};
        \coordinate (bc2) at (barycentric cs:0=1,2=1,bc=1) {};

        \foreach \i in {0, 1, 2} {
            \draw[fill] ($(bc\i)-(\i*120 + 0:0.09)$) circle (0.03);
            \draw[fill] ($(bc\i)+(\i*120 + 0:0.09)$) circle (0.03);
            \draw[fill] ($(bc\i)+(\i*120 + 90:0.09)$) circle (0.03);
        }
    \end{scope}
\end{tikzpicture}
    \end{center}
    \caption{A 2D exact Stokes complex on barycentrically refined meshes.}\label{fig:hct-complex}
\end{figure}
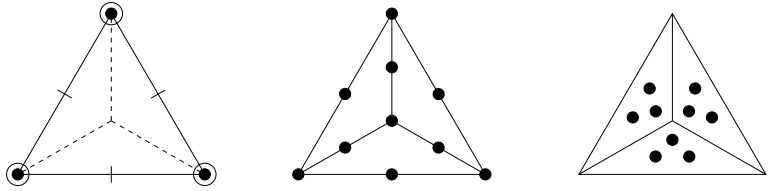
We call the mesh prior to barycentric refinement the \emph{macro mesh} and its cells \emph{macro elements}.
For a given vertex $v_i$ in the macro mesh, we define the $\MacroStar(v_i)$ of the vertex as the union of all macro elements touching the vertex\footnote{The \emph{star}
operation is a standard concept in algebraic topology \cite[\S 2]{munkres1984}; given a simplicial complex, the star of a simplex $p$ is the union of the interiors of all simplices that contain $p$. The
$\MacroStar$ operation is merely the star applied to vertices on the macro mesh.}.
We then see that for every HCT basis function $\Phib_j$ there exists a vertex $v_i$ such that $\mathrm{supp}(\Phib_j)\subset \MacroStar(v_i)$.
Hence, also $\mathrm{supp}(\nabla\times \Phib_j)\subset\MacroStar(v_i)$ and if we define
\begin{equation}
    V_i = \{\vb\in V_h : \mathrm{supp}(\vb)\subset\MacroStar(v_i)\}\label{eqn:macrostarsubspaces}
\end{equation}
then these subspaces decompose the kernel.
More recently in~\citet{fu_exact_2018} an $H^1(\curl,\Omega)$-conforming element on barycentrically refined tetrahedral meshes was introduced that forms an exact sequence with piecewise cubic continuous velocities and piecewise quadratic discontinuous pressures.
Hence, by the same argument we obtain that the $\MacroStar$ around vertices provides a decomposition of the kernel of the divergence in three dimensions.

We note that in two dimensions a quintic basis for $\Sigma_h$ is known even without macro element structure~\citep{morgan_nodal_1975}, which was used by~\citet{lee2009} to construct a robust relaxation method for sufficiently high polynomial degrees.
\subsection{Decomposing the kernel by a localised Fortin operator}
In the previous section we introduced discrete exact sequences as a tool to construct a space decomposition that also decomposes the kernel of the divergence.
When such an exact sequence exists, the approach is clearly attractive since the space decomposition can be found by simply studying the support of the basis functions in $\Sigma_h$.
However, an exact sequence only guarantees the existence of some $\Phib_h\in \Sigma_h$ so that $\nabla\times\Phib_h = \ub_h$, but does not make statements about its norm.
The proof of exactness in~\cite{fu_exact_2018} for example is based on a counting argument and hence does not provide any stability bounds.

In two dimensions, this is not an issue as it is straightforward to obtain an element in $\Sigma_h$ with bounded norm, as we will now argue.
For a divergence-free vector field $\ub_h$ we know~\citep[Theorem 3.3]{girault1986} that there exists a $\Phib\in H^2_0(\Omega)$ such that $\nabla\times \Phib = \ub_h$ and $\|\Phib\|_2 \preceq \|\ub_h\|_1$.
Since $\nabla \times$ in two dimensions simply corresponds to the
rotated gradient, we see that any two $\Phib$ that satisfy
$\nabla\times \Phib = \ub_h$ are equal up to a constant, and hence we
have in fact $\Phib\in\Sigma_h$. 

In three dimensions, the second step in this argument fails.
It was proven \review{in~\citet[Proposition~4.1]{costabel_bogovskii_2010}} that the \emph{regularised Poincar\'e} operator provides a bounded linear map
\begin{equation}
    R: [H^s(\Omega)]^3\to [H^{s+1}(\Omega)]^3 \ \text{s.t.} \ \nabla\times R(\ub) = \ub \ \text{for all} \ \ub \in H^s \ \text{with} \ \nabla \cdot \ub = 0,
\end{equation}
for any $s \in \mathbb{R}$ and domain $\Omega$ that is star-like with respect to some ball.
However, although the norm of the potential obtained from this map is bounded by the norm of $\ub$, we cannot directly infer this property for the discrete potential as the uniqueness property that we exploited in two dimensions does not hold.
In fact, we can add any gradient to the potential $\Phib$ and still preserve $\nabla\times\Phib=\ub$. This motivates the development
of a different strategy.

We briefly recall the approach of Sch\"oberl for constructing the space decomposition.
We start with a $\Phib \in H^2$ and consider a smooth partition of unity $\{\rho_i\}$.
Considering $\ub_i = \nabla \times (\rho_i \Phib)$, we immediately obtain that $\nabla \cdot \ub_i = 0$
with $\sum_i \ub_i = \ub$. However, this construction may not yield an element of the finite element space. We thus
consider $\ub_i = I_h\left(\nabla \times (\rho_i \Phib)\right)$ for an appropriate interpolation operator $I_h$
that preserves the divergence in a suitable sense.
\begin{proposition}
    Assume that $\Omega$ is star-like with respect to some ball and let $I_h : V \to V_h$ be a Fortin operator, i.e.~it satisfies
    \begin{itemize}
        \item $I_h$ is linear and continuous,
        \item $(q_h, \nabla\cdot(I_h(\vb))) = (q_h, \nabla\cdot \vb)$ for all $q_h\in Q_h$ and $\vb\in V$,
        \item $I_h(\vb_h) = \vb_h$ for $\vb_h\in V_h$.
    \end{itemize}
    Furthermore, let $\{\Omega_i\}$ be a covering of $\Omega$ with an associated smooth partition of unity $\{\rho_i\}$ satisfying
    \begin{equation}
        \begin{aligned}
            \|\rho_i\|_{L^{\infty}}&\le 1,  \\
            \|\rho_i\|_{W^{1, \infty}}&\preceq h^{-1},\\
            \|\rho_i\|_{W^{2, \infty}}&\preceq h^{-2},\\
            \mathrm{supp}(\rho_i)&\subset \Omega_i,
        \end{aligned}
    \end{equation}
    then the space decomposition $\{V_i\}$ with
    \begin{equation}
        V_i \defeq \{ I_h(v) : v\in V,\ \supp(v) \subset \Omega_i\}
    \end{equation}
    satisfies
    \begin{equation}\label{eqn:firstsplitting}
        \inf_{\substack{\ub_h=\sum \ub_i\\\ub_i\in V_i}} \sum_{i} \|\ub_i\|_{1}^2 \preceq h^{-2} \|\ub_h\|_{0}^2,
    \end{equation}
    and
    \begin{equation}\label{eqn:kernel-splitting}
        \inf_{\substack{\ub_0=\sum \ub_{0, i}\\\ub_{0,i}\in
\mathcal{N}_h \cap V_i}} \sum_{i} \|\ub_{0,i}\|_{1}^2 \preceq h^{-4} \|\ub_0\|_{0}^2.
    \end{equation}
\end{proposition}
\begin{remark}
   This is an abstraction of Sch\"oberl's approach.
   He used this strategy for a particular covering and a particular Fortin operator.
\end{remark}
\begin{remark}
    In the multigrid context the goal is to choose the covering and the Fortin operator so that the spaces $\{V_i\}$ are small, as we use direct methods to solve the problems on the spaces $V_i$.
\end{remark}
\begin{proof}
    To prove the first statement, for $\ub_h \in V_h$ we define
   \begin{equation}
        \ub_i = I_h(\rho_i \ub_h) \in V_i,   
   \end{equation} 
   and observe that
   \begin{equation}
       \sum_i \ub_i = I_h \bigg (\sum_i \rho_i \ub_h \bigg) = I_h(\ub_h) = \ub_h
   \end{equation}
   and 
   \begin{equation}
       \begin{aligned}
           \|\ub_i\|_{H^1(\Omega)}^2 & \preceq \|\rho_i \ub_h\|_{H^1(\Omega_i)}^2 \\
                         &\le \|\ub_h\|_{L^2(\Omega_i)}^2 \|\nabla \rho_i\|_{L^\infty(\Omega_i)}^2 + \|\ub_h\|_{H^1(\Omega_i)}^2 \|\rho_i\|_{L^\infty(\Omega_i)}^2 \\
                         &\preceq h^{-2} \|\ub_h\|_{L^2(\Omega_i)}^2.
       \end{aligned}
     \end{equation}
     \review{Summing over $i$ yields~\eqref{eqn:firstsplitting}.}

     For the second splitting \eqref{eqn:kernel-splitting} we first note that given a $\ub_0\in
     \Nc_h$ by \citet[Theorem 3.3]{girault1986} in 2D and by
     \review{\citet[Proposition 4.1]{costabel_bogovskii_2010}} in 3D there exists a $\Phib
     \in H^2$ such that $\nabla\times \Phib = \ub_0$ and
     $\|\Phib\|_2\preceq \|\ub_0\|_1$ and $\|\Phib\|_1 \preceq
     \|\ub_0\|_0$.

   At this point we have used the fact that discretely divergence-free vector fields are exactly divergence-free.
   This is not always necessary:
   for example in the case of the $\PtwoPzero$ element, one can modify $\ub_0$ in the interior of each cell to obtain an exactly divergence-free field.

   The estimates are obtained in a manner similar to the previous case: we define
   \begin{equation}
       \ub_{0, i} = I_h(\nabla\times (\rho_i \Phib))
   \end{equation}
   and then calculate
   \begin{equation}
       \begin{aligned}
           \|\ub_{0, i}\|_{H^1(\Omega)}^2 & \preceq \|\nabla\times(\rho_i \Phib)\|_{H^1(\Omega_i)}^2 \\
           &\preceq \|\rho_i \Phib\|_{H^2(\Omega_i)}^2 \\
           &\le \|\Phib\|_{L^2(\Omega_i)}^2 \|\nabla^2 \rho_i\|_{L^\infty(\Omega_i)}^2 + \|\Phib\|_{H^1(\Omega_i)}^2 \|\nabla \rho_i\|_{L^\infty(\Omega_i)}^2+ \|\Phib\|_{H^2(\Omega_i)}^2 \|\rho_i\|_{L^\infty(\Omega_i)}^2 \\
           &\le  h^{-4} \|\Phib\|_{L^2(\Omega_i)}^2 + h^{-2} \|\Phib\|_{H^1(\Omega_i)}^2 + \|\Phib\|_{H^2(\Omega_i)}^2.
       \end{aligned}
   \end{equation}
   Summing over $i$ and denoting the overlap by $N_O$ we obtain
   \begin{equation}
       \begin{aligned}
           \sum_i \|\ub_{0, i}\|_{\review{H^1(\Omega)}}^2 &\preceq N_O \big(h^{-4} \|\Phib\|_{L^2(\Omega)}^2 + h^{-2} \|\Phib\|_{H^1(\Omega)}^2 + \|\Phib\|_{H^2(\Omega)}^2\big)\\
                                     &\preceq N_O \big(h^{-4} \|\ub_0\|_{L^2(\Omega)}^2 + h^{-2} \|\ub_0\|_{L^2(\Omega)}^2 + \|\ub_0\|_{H^1(\Omega)}^2\big)\\
                                     &\preceq N_O h^{-4} \|\ub_0\|_{L^2(\Omega)}^2,
       \end{aligned}
   \end{equation}
   where we used an inverse inequality in the last step.
\end{proof}

Sch\"oberl uses an operator $I_h$ that is essentially the same one as in the classical proof for the inf-sup stability of the $\PtwoPzero$ element, with minor modifications so that it uses only values on element boundaries.
This leads to small subspaces $V_i$ when $\Omega_i$ is defined as a domain within the star around each vertex, see Figure~\ref{fig:alfeld-stars}.
\begin{figure}
    \begin{center}
        \begin{tikzpicture}[scale=4,
  bcdof/.style={postaction={decorate,decoration={markings,
        mark=at position 0.5 with {\draw[solid, fill, black!60] circle
          (0.12);}}}},
  dof/.style={postaction={decorate,decoration={markings,
        mark=at position 0.5 with {\draw[solid, fill, black] circle (0.12);}}}
  }]

  \foreach \i in {0, 1, 2, 3} {
      \coordinate(0-\i) at ($(\i*1.5, 0)$);
      \coordinate(1-\i) at ($(0-\i) + (60:1.5)$);
      \coordinate(2-\i) at ($(1-\i) + (120:1.5)$);

    \foreach \j in {0, 2} {
      \draw[fill] (\j-\i) circle (0.03);
    }
  }
  \foreach \i in {0, 1, 2} {
    \draw[fill] (1-\i) circle (0.03);
  }

  \foreach \i/\j in {0/1, 1/2, 2/3} {
    \foreach \k in {0, 2} {
      \draw[dof, line join=miter, thick] (\k-\i) -- (\k-\j);
    }
    \draw[dof, line join=miter, thick] (0-\i) -- (1-\i);
    \draw[dof, line join=miter, thick] (1-\i) -- (2-\i);
    \draw[dof, line join=miter, thick] (0-\j) -- (1-\i);
    \draw[dof, line join=miter, thick] (1-\i) -- (2-\j);
  }

  \foreach \i/\j in {0/1, 1/2} {
    \draw[dof, line join=miter, thick] (1-\i) -- (1-\j);
  }
  
  \foreach \i/\j in {0/1, 1/2, 2/3} {
      \coordinate (bc-0-\i-\j) at (barycentric cs:0-\i=1,0-\j=1,1-\i=1);
    \draw[bcdof, line join=miter, dotted, thin] (0-\i) -- (bc-0-\i-\j);
    \draw[bcdof, line join=miter, dotted, thin] (1-\i) -- (bc-0-\i-\j);
    \draw[bcdof, line join=miter, dotted, thin] (0-\j) -- (bc-0-\i-\j);

    \coordinate (bc-2-\i-\j) at (barycentric cs:1-\i=1,2-\i=1,2-\j=1);
    \draw[bcdof, line join=miter, dotted, thin] (1-\i) -- (bc-2-\i-\j);
    \draw[bcdof, line join=miter, dotted, thin] (2-\i) -- (bc-2-\i-\j);
    \draw[bcdof, line join=miter, dotted, thin] (2-\j) -- (bc-2-\i-\j);

    \foreach \k in {0, 2} {
      \draw[fill, black!60] (bc-\k-\i-\j) circle (0.03);
    }
  }
  \foreach \i/\j in {0/1, 1/2} { 
    \coordinate (bc-1-\i-\j)  at (barycentric cs:1-\i=1,0-\j=1,1-\j=1);
    \draw[bcdof, line join=miter, dotted, thin] (1-\i) -- (bc-1-\i-\j);
    \draw[bcdof, line join=miter, dotted, thin] (0-\j) -- (bc-1-\i-\j);
    \draw[bcdof, line join=miter, dotted, thin] (1-\j) -- (bc-1-\i-\j);

    \coordinate (bc-3-\i-\j)  at (barycentric cs:1-\i=1,1-\j=1,2-\j=1);
    \draw[bcdof, line join=miter, dotted, thin] (1-\i) -- (bc-3-\i-\j);
    \draw[bcdof, line join=miter, dotted, thin] (1-\j) -- (bc-3-\i-\j);
    \draw[bcdof, line join=miter, dotted, thin] (2-\j) -- (bc-3-\i-\j);
    \foreach \k in {1, 3} {
      \draw[fill, black!60] (bc-\k-\i-\j) circle (0.03);
    }
  }

  \begin{scope}[on background layer]
      \draw[blue!50, fill=blue!50] ([shift={(0:5pt)}]1-0) --
      ([shift={(30:5pt)}] bc-1-0-1) --
      ([shift={(60:5pt)}]0-1) --
      ([shift={(90:5pt)}] bc-0-1-2) --
      ([shift={(120:5pt)}]0-2) --
      ([shift={(150:5pt)}] bc-1-1-2) --
      ([shift={(180:5pt)}]1-2) --
      ([shift={(210:5pt)}] bc-3-1-2) --
      ([shift={(240:5pt)}]2-2) --
      ([shift={(270:5pt)}] bc-2-1-2) --
      ([shift={(300:5pt)}]2-1) --
      ([shift={(330:5pt)}] bc-3-0-1) --
      cycle;

      \draw[blue!50, fill=blue!50] ([shift={(330:5pt)}]2-0) --
      ([shift={(210:5pt)}]2-1) --
      ([shift={(90:5pt)}]1-0) --
      cycle;
  \end{scope}
\end{tikzpicture}
    \end{center}
    \caption{The two types of patches obtained by taking the star around vertices in a barycentrically refined mesh.}\label{fig:alfeld-stars}
\end{figure}
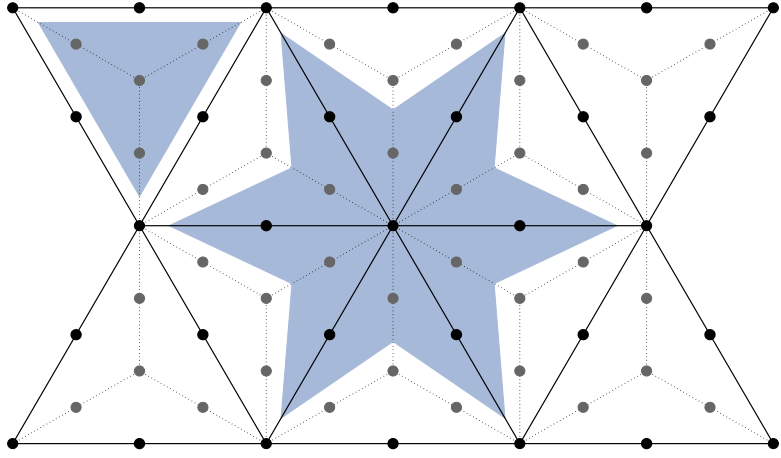
However, as of writing the authors are not aware of a similar construction of a Fortin operator for the Scott--Vogelius element in either two or three dimensions that could be modified.\footnote{\review{During the review process we were made aware of recent work \cite[Theorem 3.6]{Boffi21} in which a Fortin operator is constructed in two dimensions for a similar problem using a rotated gradient and different boundary conditions.}} 
To obtain a Fortin operator with the needed locality, we instead use
that the \review{finite} element \review{pair} is known to be inf-sup stable on a single macro
element:
we begin by considering a Fortin operator $\tilde I_h$ that preserves the divergence with respect to pressures that are piecewise constant on macro elements and then enrich this operator
with local Fortin operators on each macro element.

Given a domain $\Omega$, we consider a simplicial mesh $\mesh_h = \{K^h\}$ with $\bigcup_{K^h\in\mesh_h} K^h = \overline{\Omega}$ and $(K^h_1)^\circ \cap (K^h_2)^\circ = \emptyset$ for distinct $K^h_1, K^h_2 \in \mesh_h$.
The elements $K^h\in\mesh_h$ will be referred to as the \emph{macro cells}.

\begin{lemma} \label{lem:glue-fortin}
    Let $\tilde Q_h$ be defined by
    \begin{equation}
        \tilde Q_h \defeq \{ \tilde q_h \in L^2(\Omega): \tilde q_h \vert_K \equiv \mathrm{const}\,\text{ for all } K \in \mesh_h\}
    \end{equation}
    and assume that there exists a Fortin operator $\tilde I_h:V \to V_h$ for the pair $V_h\times \tilde Q_h$, that is
    \begin{itemize}
        \item $\tilde I_h$ is linear and continuous,
        \item $(\tilde q_h, \nabla\cdot(\tilde I_h(\vb))) = (\tilde q_h, \nabla\cdot \vb)$ for all $\tilde q_h\in \tilde Q_h$ and $\vb\in V$,
        \item $\tilde I_h(\vb_h) = \vb_h$ for $\vb_h\in V_h$,
        \item there exists a covering $\{\Omega_i\}$ of $\Omega$ such that for all vertices $v_i$ in the macro mesh it holds that $\Omega_i \subset \MacroStar(v_i)$ and
            \begin{equation}
                \tilde I_h(\vb) \in V_i \text{ for all }\vb \in V \text{ such that } \mathrm{supp}(\vb)\subset \Omega_i
            \end{equation}
            where $\{V_i\}$ are the subspaces defined in~\eqref{eqn:macrostarsubspaces}, that is
            \begin{equation}
                V_i = \{\vb\in V_h : \mathrm{supp}(\vb)\subset\MacroStar(v_i)\}.
            \end{equation}
    \end{itemize}
    Furthermore assume that the \review{finite} element \review{pair} is stable on each macro cell.
    Then there exists a linear map $I_h : V\to V_h$ such that
    \begin{itemize}
        \item $I_h$ is linear and continuous,
        \item $(q_h, \nabla\cdot(I_h(\vb))) = (q_h, \nabla\cdot \vb)$ for all $q_h\in Q_h$ and $\vb\in V$,
        \item $I_h(\vb_h) = \vb_h$ for $\vb_h\in V_h$,
        \item the covering $\{\Omega_i\}$ has the property that for all vertices $v_i$ in the macro mesh it holds
            \begin{equation}\label{eqn:fortin-locality}
                I_h(\vb) \in V_i \text{ for all }\vb \in V \text{ such that } \mathrm{supp}(\vb)\subset \Omega_i.
            \end{equation}
    \end{itemize}
\end{lemma}

\begin{remark}
        The key difference between $\tilde I_h$ and $I_h$ is that the former only has to preserve the divergence with respect to test functions in the smaller space $\tilde Q_h$, but $I_h$ preserves the divergence with respect to test functions in the full space $Q_h$.
\end{remark}

\begin{remark}
        The advantageous consequence of this result is that the macro star gives a kernel-capturing space decomposition
with small support, suitable for use as a multigrid relaxation method.
\end{remark}

Before proving this statement, we give two examples.

For the first example, consider the Scott--Vogelius element on Alfeld splits as shown in Figure~\ref{fig:scott-zhang-pou}.
By~\citet[Section 4.6]{qin1994} and~\citet{zhang2004} the \review{finite} element \review{pair} is inf-sup stable for $k=d$ in both two and three dimensions.

We begin by considering the two dimensional case.
The covering $\{\Omega_i\}$ is shown as the blue shaded region in the figure and the operator $\tilde I_h$ can be chosen to be the standard Fortin operator for the $\PtwoPzero$ finite element pair,
constructed as follows.
Let $I_1$ be a Scott--Zhang interpolant based on the integration domains shown in red in Figure~\ref{fig:scott-zhang-pou}. First, note that $I_1(\vb_h) = \vb_h$ for all $\vb_h \in V_h$, and
that by construction $I_1(\vb) \in V_i$ for all $\vb$ with support in $\Omega_i$. Let $I_2: V \to V_h$ be defined to be zero on all degrees of freedom except the degrees of freedom on the macro edges,
which instead are chosen so that
\begin{equation}\label{eqn:flux-edge}
\int_E I_2(\vb) \dr s = \int_E \vb \dr s,
\end{equation}
for all edges $E$ in the macro mesh. Then let $\tilde I_h(\vb) = I_1(\vb) + I_2(\vb - I_1(\vb))$. By construction,
\begin{equation}
(\nabla \cdot I_2(\vb), \tilde q_h) = (\nabla \cdot \vb, \tilde q_h) \ \text{ for all } \tilde q_h \in \tilde Q_h, \vb \in V,
\end{equation}
and hence the same property holds for $\tilde I_h$. The existence of a global Fortin operator with the required locality properties then
follows from Lemma~\ref{lem:glue-fortin}.
Comparing Figures~\ref{fig:alfeld-stars} and~\ref{fig:scott-zhang-pou}
we note that the subspaces obtained here are larger but that fewer are
obtained, since we only obtain one subspace per vertex in the mesh
prior to barycentric refinement.
\begin{figure}[htbp]
    \begin{center}
        \begin{tikzpicture}[scale=4,
  bcdof/.style={postaction={decorate,decoration={markings,
        mark=at position 0.5 with {\draw[solid, fill, black!60] circle
          (0.12);}}}},
  dof/.style={postaction={decorate,decoration={markings,
        mark=at position 0.5 with {\draw[solid, fill, black] circle (0.12);}}}
  }]

  \foreach \i in {0, 1, 2, 3} {
    \coordinate (0-\i) at ($(\i*1.5, 0)$);
    \coordinate (1-\i) at ($(0-\i) + (60:1.5)$);
    \coordinate (2-\i) at ($(1-\i) + (120:1.5)$);

    \foreach \j in {0, 2} {
      \draw[fill] (\j-\i) circle (0.03);
    }
  }
  \foreach \i in {0, 1, 2} {
    \draw[fill] (1-\i) circle (0.03);
  }

  \foreach \i/\j in {0/1, 1/2, 2/3} {
    \foreach \k in {0, 2} {
      \draw[dof, line join=miter, thick] (\k-\i) -- (\k-\j);
    }
    \draw[dof, line join=miter, thick] (0-\i) -- (1-\i);
    \draw[dof, line join=miter, thick] (1-\i) -- (2-\i);
    \draw[dof, line join=miter, thick] (0-\j) -- (1-\i);
    \draw[dof, line join=miter, thick] (1-\i) -- (2-\j);
  }

  \foreach \i/\j in {0/1, 1/2} {
    \draw[dof, line join=miter, thick] (1-\i) -- (1-\j);
  }
  
  \foreach \i/\j in {0/1, 1/2, 2/3} {
    \coordinate (bc-0-\i-\j) at (barycentric cs:0-\i=1,0-\j=1,1-\i=1);
    \draw[bcdof, line join=miter, dotted, thin] (0-\i) -- (bc-0-\i-\j);
    \draw[bcdof, line join=miter, dotted, thin] (1-\i) -- (bc-0-\i-\j);
    \draw[bcdof, line join=miter, dotted, thin] (0-\j) -- (bc-0-\i-\j);

    \coordinate (bc-2-\i-\j) at (barycentric cs:1-\i=1,2-\i=1,2-\j=1);
    \draw[bcdof, line join=miter, dotted, thin] (1-\i) -- (bc-2-\i-\j);
    \draw[bcdof, line join=miter, dotted, thin] (2-\i) -- (bc-2-\i-\j);
    \draw[bcdof, line join=miter, dotted, thin] (2-\j) -- (bc-2-\i-\j);

    \foreach \k in {0, 2} {
      \draw[fill, black!60] (bc-\k-\i-\j) circle (0.03);
    }
  }
  \foreach \i/\j in {0/1, 1/2} { 
    \coordinate (bc-1-\i-\j)  at (barycentric cs:1-\i=1,0-\j=1,1-\j=1);
    \draw[bcdof, line join=miter, dotted, thin] (1-\i) -- (bc-1-\i-\j);
    \draw[bcdof, line join=miter, dotted, thin] (0-\j) -- (bc-1-\i-\j);
    \draw[bcdof, line join=miter, dotted, thin] (1-\j) -- (bc-1-\i-\j);

    \coordinate (bc-3-\i-\j)  at (barycentric cs:1-\i=1,1-\j=1,2-\j=1);
    \draw[bcdof, line join=miter, dotted, thin] (1-\i) -- (bc-3-\i-\j);
    \draw[bcdof, line join=miter, dotted, thin] (1-\j) -- (bc-3-\i-\j);
    \draw[bcdof, line join=miter, dotted, thin] (2-\j) -- (bc-3-\i-\j);
    \foreach \k in {1, 3} {
      \draw[fill, black!60] (bc-\k-\i-\j) circle (0.03);
    }
  }

  \begin{scope}[on background layer]
      \draw[blue!50, fill=blue!50] ([shift={(0:5pt)}]1-0) --
      ([shift={(60:5pt)}]0-1) --
      ([shift={(120:5pt)}]0-2) --
      ([shift={(180:5pt)}]1-2) --
      ([shift={(240:5pt)}]2-2) --
      ([shift={(300:5pt)}]2-1) --
      cycle;
      \draw[line width=3.5pt, line cap=round, red] (barycentric cs:bc-3-0-1=1,2-1=1) -- +(0, -3pt);

      \draw[fill, red] (2-1) circle (0.1);

      \draw[line width=3.5pt, line cap=round, red] (barycentric cs:2-1=1,1-1=1) ++(120:3pt) -- +(-60:6pt);

      \draw[fill, red] (bc-3-0-1) circle (0.1);
  \end{scope}
\end{tikzpicture}
    \end{center}
    \caption{Domain $\Omega_i$ around a vertex $v_i$ (blue), covering all degrees of freedom inside $\MacroStar(v_i)$. The figure also shows integration regions (in red) for vertex and edge degrees of freedom used in the Scott--Zhang operator $I_1$. Note that the integration regions are chosen so that only those associated with $\MacroStar(v_i)$ intersect with $\Omega_i$.}\label{fig:scott-zhang-pou}
\end{figure}
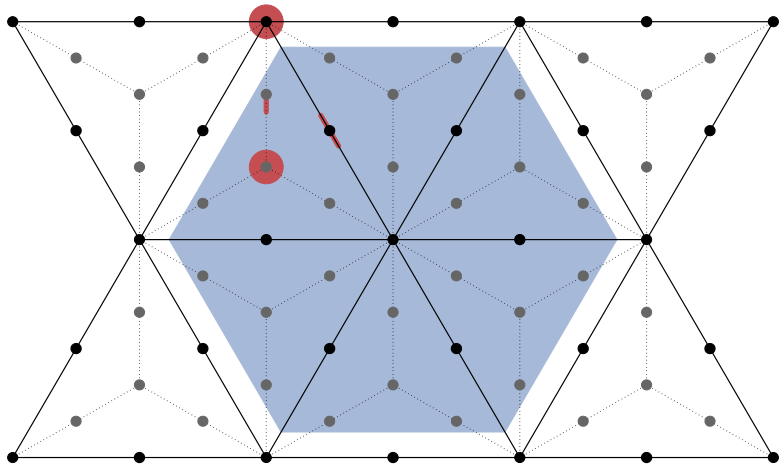

Essentially the same construction can be used in three dimensions.
$I_1$ is again defined as a Scott--Zhang interpolant, $I_2$ is set to be zero on all vertex and edge degrees of freedom, and the value on each facet $F$ is chosen so that
\begin{equation}\label{eqn:flux-facet}
    \int_F I_2(\vb) \dr x = \int_F \vb \dr x \quad \text{ for all } \vb \in V.
\end{equation}

Finally we note that a similar statement holds for the second example of the Powell--Sabin split, shown in Figure~\ref{fig:powell-sabin}.
\review{Due to the presence of singular vertices/edges, one has to consider a slightly smaller pressure space $Q_h = \{\nabla \cdot \vb_h : \vb_h \in V_h\}$, but then} the element pair $V_h\times Q_h$ is inf-sup stable for $k=d-1$~\citep{zhang_p1_2008,zhang_quadratic_2011, guzman_exact_2020}.
We highlight that the Powell--Sabin split introduces vertices on the macro edges in 2D and on the macro facets in 3D.
The degrees of freedom on these vertices are used in the construction of $I_2$ so that \eqref{eqn:flux-edge} holds in 2D and \eqref{eqn:flux-facet} holds in 3D respectively.
\review{Finally,  we note that to implement the smoother one does not need an explicit description of $Q_h$, but only of $V_h$, so the more complicated nature of $Q_h$ for these splits does not raise any practical issues.}
\begin{remark}
    \review{
    When compared using elements of the same degree, Alfeld splits are more efficient, as the $\MacroStar(v_i)$ for that split contains fewer degrees of freedom and hence the local solves are cheaper.
    However, Powell--Sabin splits enable the use of lower order discretisations.
    The lower degree offered by Powell--Sabin splits could make them more attractive when additional singular terms are present in the equations that are captured by the support of $C^1$ elements, such as
    when interior penalty stabilisation is used for the Navier--Stokes equations; see~\citet[Remark 4.1]{FMSW21} for an example.
}
\end{remark}
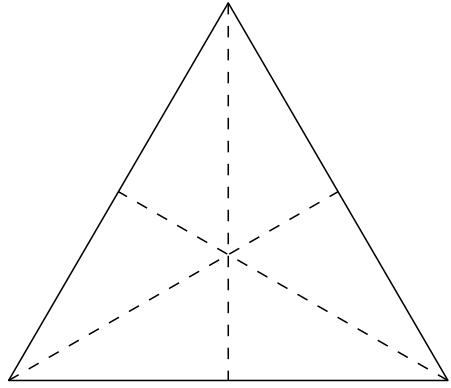
\begin{figure}[htbp]
    \begin{center}
        \ExplSyntaxOn
\newcommand{\fpeval}[1]{\fp_eval:n{#1}}
\ExplSyntaxOff

\begin{tikzpicture}[scale=4]
    \begin{scope}[shift={(1.5,0)}]
        \coordinate (0) at (0, 0);
        \coordinate (1) at (1, 0);
        \coordinate (2) at (0.5, \fpeval{sqrt(3)/2});
        \draw[line join=round] (0) -- (1) -- (2) -- cycle;
        \draw[line join=round, dashed] (0) -- (barycentric cs:1=1,2=1);
        \draw[line join=round, dashed] (1) -- (barycentric cs:0=1,2=1);
        \draw[line join=round, dashed] (2) -- (barycentric cs:0=1,1=1); \coordinate (0) at (0, 0);
    \end{scope}
\end{tikzpicture}
    \end{center}
    \caption{Powell--Sabin split in two dimensions.}\label{fig:powell-sabin}
\end{figure}

\begin{proof}[Proof of Lemma \ref{lem:glue-fortin}]
    The idea is to combine the global Fortin operator $\tilde I_h$ that preserves the discrete divergence with respect to pressures that are constant on macro cells with suitable local Fortin operators.


    To treat the divergence with respect to the remaining pressures in $Q_h\setminus \tilde Q_h$ we now consider the macro elements separately.
    For each such macro triangle $K$ we define the spaces $V_{h, 0}(K) = \{\vb_h\vert_K:\vb_h \in V_h,\ \mathrm{supp}(\vb_h)\subset K\}$ and $Q_h(K)=\{q_h\vert_K: q_h \in \tilde Q_h^\perp\}$.
    The space $V_{h, 0}(K)$ consists of velocity fields supported in a macro cell and the space $Q_h(K)$ consists of pressures on a macro cell that integrate to zero.
    We note that $Q_h = \sum_{K\in \mesh_h} Q_h(K) \oplus \tilde Q_h $.
    Since we assume that this pair is inf-sup stable for each macro element, we know that there exist Fortin operators $I_h^K: V(K)\to V_{h,0}(K)$, where $V(K) = \{\vb\vert_K: \vb\in V\}$, such that
    \begin{itemize}
        \item $I_h^K(\vb_h)=\vb_h$ for $\vb_h\in V_{h,0}(K)$
        \item $I_h^K$ is bounded as a map $V(K)\mapsto V_{h, 0}(K)$
        \item $(q_h, \nabla\cdot\vb) = (q_h, \nabla\cdot(I_h^K(\vb)))$ for all $q_h \in Q_h(K)$
    \end{itemize}
    for all $K$ (\citet{fortin_analysis_1977},~\citet[Lemma 4.19]{ern2004}).
    
    Now we define
    \begin{equation}
        I_h(\vb) = \tilde I_h(\vb) + \sum_K I_h^K( (\vb-\tilde I_h(\vb))\vert_K).
    \end{equation}
    Clearly, $I_h$ is linear and $I_h(\vb_h)=\vb_h$ for all $\vb_h \in V_h$.
    In addition, $I_h$ is continuous with continuity constant only dependent on the continuity constant of $\tilde I_h$ and the local Fortin operators, and $I_h$ satisfies the locality property in~\eqref{eqn:fortin-locality}. 

    Furthermore, we note that the discrete divergence of vector fields in $V_{h,0}(K)$ with respect to $\tilde Q_h$ is zero.
    It follows that
    \begin{equation}
        \begin{aligned}
            & (q_h, \nabla\cdot(I_h(\vb)))\\
        = & \xunderbrace{(q_h, \nabla\cdot(\tilde I_h(\vb)))}_{=(q_h, \nabla\cdot\vb)} + \smash{\sum_K \xunderbrace{(q_h, \nabla \cdot (\xoverbrace{I_h^K( (\vb-\tilde I_h(\vb))\vert_K)}^{\in V_{h,0}(K)}))}_{=0}} \\
            = & (q_h, \nabla\cdot\vb)
        \end{aligned}
    \end{equation}
    for all $q_h \in \tilde Q_h$ and $ \vb\in V$.

    Lastly, we show that $I_h$ preserves the discrete divergence with respect to the local pressures in $Q_h(K)$.
    For $\vb\in V$, $K\in \mesh_h$, and $q_h \in Q_h(K)$, we have
    \begin{equation}
        \begin{aligned}
            &(q_h, \nabla\cdot(I_h(\vb)))\\
            ={}&(q_h, \nabla\cdot(\tilde I_h(\vb))) + (q_h, \nabla\cdot(I_h^K((\vb-\tilde I_h\vb)\vert_K)))\\
            ={}&(q_h, \nabla\cdot(\tilde I_h(\vb))) + (q_h, \nabla\cdot(\vb-\tilde I_h(\vb)))\\
            ={}&(q_h, \nabla\cdot\vb),
        \end{aligned}
    \end{equation}
    as desired.
\end{proof}

\section{Prolongation}\label{sec:prolongation}

The second key ingredient for a robust multigrid scheme is a robust prolongation operator \review{$\tilde P_H$} that
maps coarse grid functions to fine grid functions with a continuity constant independent of $\gamma$.
To build intuition for this requirement, \review{let $P_H$ be a given prolongation operator and} calculate
\begin{equation}
    \begin{aligned}
        \|\ub_H \|_{A_\Hg}^2 &= \| \ub_H\|_{A_H}^2 + \gamma \|\nabla \cdot \ub_H\|_{0}^2\\
        \|P_H \ub_H \|_{A_\hg}^2 &= \|P_H \ub_H\|_{A_h}^2 + \gamma \|\nabla \cdot ( P_H\ub_H)\|_{0}^2.
    \end{aligned}
\end{equation}
The key difficulty lies in the second term of this norm.
If $V_H \not\subset V_h$, interpolation is not exact. Hence
for a divergence-free vector field $\ub_H \in \mathcal{N}_H$ the second term in
$\|\ub_H\|_{A_\Hg}^2$ vanishes, but it may not hold that $P_H \ub_H$ is divergence-free,
and so the corresponding term in $\|P_H\ub_H\|_{A_\hg}^2$ might be
large. Hence the continuity constant of the prolongation operator $P_H$ in the energy norm is not independent
of $\gamma$. 

Some macro structures cause non-nestedness, and some do not.
For example, Alfeld splits induce a non-nested mesh
structure, as seen in Figure \ref{fig:macro-hierarchy}.
\review{On the other hand, in two dimensions and on regular meshes, the
additional nodes induced by regular refinement are a subset of those induced by
the Powell--Sabin split and the resulting hierarchy is nested, again demonstrated in Figure~\ref{fig:macro-hierarchy}.}
The work of
\citet{lee2009} considered uniform refinements,
which are nested, and hence they do not require any
modification of the prolongation operator. For the remainder of
this section we consider macro meshes that induce non-nested
hierarchies, such as in the case of Alfeld splits.

\begin{figure}
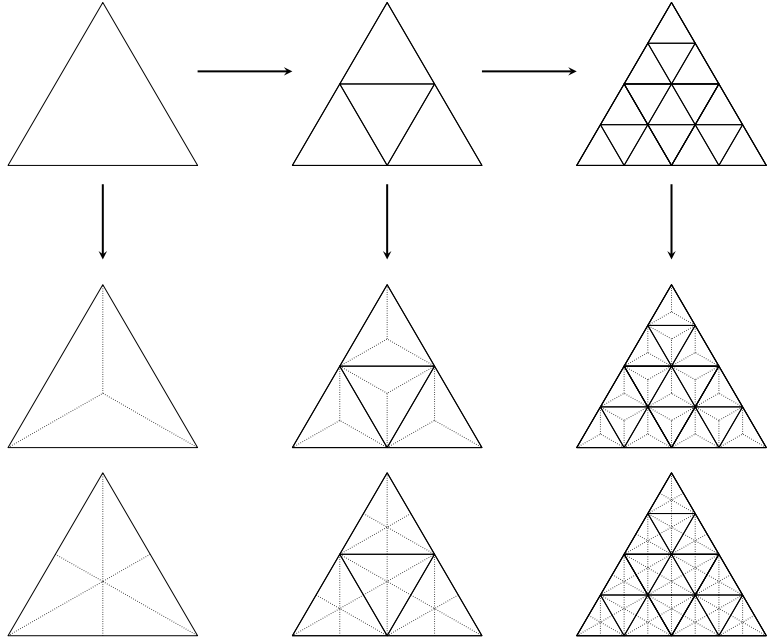

  \begin{center}
      \includestandalone[width=0.7\textwidth]{./figures/macro-hierarchy}
    \caption{A three level multigrid hierarchy \review{using either Alfeld splits or Powell--Sabin splits at each level}.}\label{fig:macro-hierarchy}
    \end{center}
\end{figure}

To remedy the $\gamma$-dependence of the standard
prolongation operator, we must modify it to map fields that are
divergence-free on the coarse grid to fields that are (nearly)
divergence-free on the fine grid, i.e.~that $\|\nabla \cdot ( \tilde P_H\ub_H)\|_{0}^2$ is of order $\mathcal{O}(\gamma^{-1})$.
We now describe a modification of the standard prolongation operator that satisfies this condition.
This type of modification goes back to Sch\"oberl's work, although we give a different derivation and proof.

Let $\ub_H \in \Nc_H$ be a divergence-free function on the coarse-grid and denote the standard prolongation induced by the interpolation operator on the finite element space by $P_H \ub_H$.
We are interested in finding a small perturbation $\tilde \ub_h$ such that \review{$\tilde P_H u_H = P_H \ub_H - \tilde \ub_h\in \Nc_h$}.
This could be achieved by solving
\begin{equation}
    \begin{aligned}
        &\min_{\tilde \ub_h \in V_h} && a(\tilde \ub_h, \tilde \ub_h)\\
        &\text{s.t.} && \Pi_{Q_h} (\nabla\cdot \tilde \ub_h) = \Pi_{Q_h} (\nabla\cdot P_H \ub_H).
    \end{aligned}
\end{equation}
This corresponds to solving a Stokes-like problem in $V_h\times Q_h$.
We now relax this problem in two ways.
First, we do not need to enforce that $P_H\ub_H - \tilde
\ub_h$ has zero divergence, as it is enough if it is suitably small, i.e.~we can instead find $\tilde \ub_h\in V_h$ that minimises
\begin{equation}
    \min_{\tilde \ub_h \in V_h} a(\tilde \ub_h, \tilde \ub_h) + \gamma \|\Pi_{Q_h} (\nabla\cdot(P_H\ub_H - \tilde \ub_h))\|_0^2.
\end{equation}
This corresponds to: find $\tilde \ub_h \in V_h$ such that
\begin{equation}\label{eqn:prolongation-global-problem}
    a_{\hg}(\tilde \ub_h, \vb_h) = \gamma (\Pi_{Q_h} (\nabla\cdot P_H \ub_H), \Pi_{Q_h} (\nabla\cdot\vb_h)) \quad \text{for all } \vb_h \in V_h.
\end{equation}
Clearly at this stage we have not gained much, since we now need to solve a global problem involving the nearly singular bilinear form $a_\hg$. (Recall that the projection onto $Q_h$ of the divergence is the identity, as $\nabla \cdot V_h = Q_h$.)
However, it turns out that under certain assumptions to be stated in the following proposition, one can instead solve the same problem on smaller spaces $\hat V_h\subset V_h$ and $\hat Q_h\subset Q_h$:
\begin{equation}
    \min_{\tilde \ub_h \in \hat V_h} a(\tilde \ub_h, \tilde \ub_h) + \gamma \|\Pi_{\hat Q_h} (\nabla\cdot (P_H\ub_H - \tilde \ub_h))\|_0^2,
\end{equation}
or equivalently: find $\tilde \ub_h\in \hat V_h$ such that
\begin{equation}\label{eqn:prolongation-local-problem-tildeqh}
    a_{\hg}(\tilde \ub_h, \hat \vb_h) = \gamma (\Pi_{\hat Q_h} (\nabla\cdot (P_H \ub_H)), \Pi_{\hat Q_h} (\nabla\cdot \hat \vb_h)) \quad \text{for all } \hat \vb_h \in \hat V_h.
\end{equation}
How exactly one chooses these subspaces will depend on the discretisation under consideration and we will again illustrate this for the Scott--Vogelius element on Alfeld splits.

\begin{proposition}[Robust prolongation]\label{prop:robust-prolongation}
    Assume we can split $Q_h = \tilde Q_H \oplus \hat Q_h$ and that $\tilde Q_H \subseteq Q_H$.
    Let $P_H: V_H\to V_h$ be a prolongation operator that is continuous in the $\|\cdot \|_1$ norm and preserves the divergence with respect to $\tilde Q_H$, i.e.
    \begin{equation}\label{eqn:default-prolongation-divergence-preserving}
       (\nabla\cdot(P_H \vb_H), \tilde q_H) = (\nabla\cdot \vb_H, \tilde q_H) \qquad \text{for all } \tilde q_H\in \tilde Q_H, \vb_H \in V_H.
    \end{equation}
    Assume in addition that there exists a $\hat V_h\subset V_h$ such that 
    \begin{equation}\label{eqn:pertubation-vanishes}
       (\nabla\cdot\hat \vb_h, \tilde q_H) = 0 \qquad \text{for all } \tilde q_H\in \tilde Q_H, \hat \vb_h \in \hat V_h,
    \end{equation}
    and such that the pairing $\hat V_h\times \hat Q_h$ is inf-sup stable,~i.e.
    \begin{equation}
        \adjustlimits \inf_{\hat q_h\in \hat Q_h} \sup_{\hat \vb_h\in \hat V_h} \frac{(\hat q_h, \nabla\cdot \hat \vb_h)}{\|\hat \vb_h\|_1 \|\hat q_h\|_0} \ge c
    \end{equation}
    for some mesh independent $c>0$.
    For $\ub_H \in V_H$, define $\tilde \ub_h$ as the solution to
    \begin{equation}\label{eqn:prolongation-local-problem}
        a_{\hg}(\tilde \ub_h, \hat \vb_h) = \gamma (\Pi_{Q_h} (\nabla\cdot (P_H\ub_H)), (\Pi_{Q_h} (\nabla\cdot (\hat \vb_h)))) \qquad \text{for all }\hat \vb_h\in \hat V_h.
    \end{equation}
    Then the prolongation $\tilde P_H:V_H\to V_h$ defined by
    \begin{equation}
        \tilde P_H \ub_H = P_H \ub_H - \tilde \ub_h
    \end{equation}
    is continuous in the energy norm with continuity constant independent of $\gamma$.
\end{proposition}
\begin{remark}
    The problems in~\eqref{eqn:prolongation-local-problem-tildeqh} and~\eqref{eqn:prolongation-local-problem} are equivalent by the assumption in~\eqref{eqn:pertubation-vanishes}.
\end{remark}
\begin{remark}
This prolongation operator is very similar to the one used by~\citet[Theorem 4.2]{schoberl1999b} and~\citet[Lemma 5.1]{benzi2006}.
The difference is that in Sch\"oberl's work the problem in~\eqref{eqn:prolongation-local-problem} is replaced with
\begin{equation}\label{eqn:prolongation-local-problem-schoeberl}
    a_{\hg}(\tilde \ub_h, \hat v_h) = a_\hg(P_H \ub_H, \hat \vb_h) \qquad \text{for all }\hat \vb_h\in \hat V_h.
\end{equation}
\review{We note that this problem only differs in the right-hand side, and hence the two methods have essentially the same cost.}
We use the version in~\eqref{eqn:prolongation-local-problem} for two reasons: first, it reduces to the standard prolongation for $\gamma=0$, and second, it can be viewed as a local version of the global problem in~\eqref{eqn:prolongation-global-problem}, whereas solving a global version of~\eqref{eqn:prolongation-local-problem-schoeberl} would lead to $\tilde \ub_h = P_H \ub_H$ and hence $\tilde P_H \ub_H = 0$.

The proofs of Sch\"oberl and Benzi \& Olshanskii are based on an equivalent mixed problem.
We will give a different proof motivated by the formulation as an optimisation problem and use the existence of a Fortin operator arising from inf-sup stability.
\end{remark}

Before giving the proof, we again consider the case of the Scott--Vogelius element on Alfeld splits.
In this case, the hierarchy is constructed as shown in Figure~\ref{fig:macro-hierarchy}.
Due to the barycentric refinement at each level, we do not have nested function spaces and hence the prolongation is not exact and a divergence-free function on the coarse grid may be prolonged to a function on the fine grid with nonzero divergence.
However, we observe that interpolation is exact on the boundaries of coarse grid macro cells.
This means that flux across these boundaries is preserved and hence 
the divergence with respect to functions in
\begin{equation}
    \tilde Q_H \coloneqq \{q \in L^2: q\equiv \mathrm{const} \text{ on coarse grid macro cells } K\in \mesh_H\},
\end{equation}
is preserved, i.e.~we have for the standard interpolation operator
$P_H: V_H \to V_h$ that
\begin{equation}
    (\nabla \cdot \ub_H, \tilde q_H) = (\nabla\cdot (P_H \ub_H), \tilde q_H) \quad \text{for all } \ub_H\in V_H,\ \tilde q_H \in \tilde Q_H.
\end{equation}
Hence requirement~\eqref{eqn:default-prolongation-divergence-preserving} of Proposition~\ref{prop:robust-prolongation} is satisfied.
To correct for any extra divergence gained due to interpolation inside a coarse grid macro cell we will solve local problems.
To this end, we define the spaces
\begin{equation}
    \begin{aligned}
        \hat Q_h &\defeq \{q_h\in Q_h : \Pi_{\tilde Q_H} q_h = 0\}\\
        \hat V_h &\defeq \{\vb_h\in V_h : \supp(\vb_h) \subset K \text{ for some } K\in \mesh_H\}.
    \end{aligned}
\end{equation}
The space $\hat V_h$ consists of local patches of velocity degrees of freedom contained in coarse grid macro cells, as shown in Figure~\ref{fig:sv-prolongation}.
We highlight that these patches decouple and hence solves involving $\hat V_h$ can be performed independently, leading to good performance.
\begin{figure}[htpb]
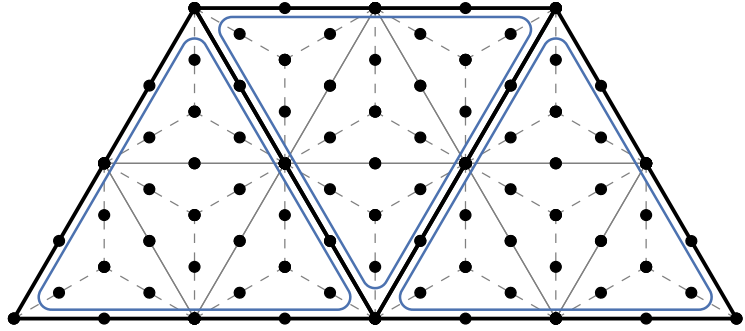

    \begin{center}
        \includestandalone[width=10cm]{./figures/baryprolongation}
    \end{center}
    \caption{Degrees of freedom of the $\hat V_h$ on which we perform local solves to obtain a robust prolongation in two dimensions.}\label{fig:sv-prolongation}
\end{figure}
The pair $\hat V_h\times \hat Q_h$ is inf-sup stable in both two and three dimensions, so it remains to check
\begin{equation}
    (\nabla\cdot \hat \vb_h, \tilde q_H) = 0 \quad \text{for all }\tilde q_H\in \tilde Q_H, \hat \vb_h\in \hat V_h,
\end{equation}
which follows from the requirement that the support of vector fields $\hat \vb_h\in \hat V_h$ is contained in coarse grid macro cells and the definition of $\tilde Q_H$.
A robust prolongation operator can now be constructed as in Proposition~\ref{prop:robust-prolongation}.

\begin{proof}[Proof of Proposition~\ref{prop:robust-prolongation}]
    We denote
    \begin{equation}
        J(\tilde \vb_h) \defeq a(\tilde \vb_h, \tilde \vb_h)  + \gamma \|\Pi_{\hat Q_h} (\nabla\cdot(P_H\ub_H-\tilde \vb_h))\|_0^2,
    \end{equation}
    and observe that $\tilde \ub_h$ is the unique minimiser of $J$ in $\hat V_h$.
    By inf-sup stability of the pairing $\hat V_h\times\hat Q_h$ there exists a continuous Fortin operator $I:V\to \hat V_h$ that satisfies
    \begin{equation}
        \Pi_{\hat Q_h} (\nabla\cdot (I\vb)) = \Pi_{\hat Q_h} (\nabla\cdot \vb) \quad \text{for all }\vb \in V.
    \end{equation}
    Let $\bar \ub_h \defeq I(P_H\ub_H)\in \hat V_h$ and observe that
    \begin{equation}\label{eqn:jtildebound}
        J(\tilde \ub_h) \le J(\bar \ub_h) = a (\bar \ub_h, \bar \ub_h) \preceq \|P_H\ub_H\|_1^2 \preceq \|\ub_H\|_1^2.
    \end{equation}
    We conclude
    \begin{equation}
        \begin{aligned}
&\|\tilde P_H\ub_H \|_{A_\hg}^2 \\
            \le& \|P_H \ub_H-\tilde \ub_h \|_{A_h}^2 + \gamma \|\Pi_{Q_h} (\nabla\cdot (P_H\ub_H -\tilde \ub_h))\|_0^2\\
            \preceq& \xunderbrace{\|P_H \ub_H\|_{A_h}^2}_{\preceq \|\ub_H\|_{A_h}^2} + \xunderbrace{\|\tilde \ub_h\|_{A_h}^2 + \gamma \|\Pi_{\hat Q_h} (\nabla\cdot(P_H\ub_H -\tilde \ub_h))\|_0^2}_{=J(\tilde \ub_h)\underset{\eqref{eqn:jtildebound}}{\preceq} \|\ub_H\|_1^2} + \gamma \|\Pi_{\tilde Q_H}(\nabla\cdot(P_H \ub_H - \tilde \ub_h))\|_0^2\\
            \preceq& \|\ub_H\|_1^2 + \gamma \|\xunderbrace{\Pi_{\tilde Q_H}(\nabla\cdot(P_H \ub_H))}_{\underset{\eqref{eqn:default-prolongation-divergence-preserving}}=\Pi_{\tilde Q_H} (\nabla\cdot \ub_H)}\|_0^2 +\gamma \|\xunderbrace{\Pi_{\tilde Q_H} (\nabla\cdot \tilde \ub_h)}_{\underset{\eqref{eqn:pertubation-vanishes}}=0}\|_0^2\\
            \preceq& \|\ub_H\|_1^2 + \gamma \|\Pi_{\tilde Q_H} (\nabla\cdot \ub_H)\|_0^2\\
                                          \preceq& \|\ub_H\|_{A_\Hg}^2.
        \end{aligned}
    \end{equation}
\end{proof}

\section{Multigrid convergence}\label{sec:mgrid}
We have seen that the bounds in the estimates for the relaxation behave like $c_1\sim h^{-2}$ and $c_2\sim h^{-4}$.
The reason for the fast growth of $c_2$ is that the splitting into divergence-free functions is constructed by expressing $\ub_0 =\nabla\times\Phib$ for a potential $\Phib \in H^2$,
introducing second derivatives into the estimates.
This quartic growth of $c_2$ complicates a standard multigrid analysis, as one needs to show that a coarse-grid solve reduces the error in such a way that the multigrid scheme yields mesh-independent convergence.
Following the structure of the proof by Sch\"oberl for the \PtwoPzero element~\cite[\S 4.4, \S 4.6.1]{schoberl1999b}, with suitable modifications, one can obtain a parameter-robust multigrid convergence result.
\begin{theorem}
Define a multigrid method using the relaxation defined in Section
\ref{sec:smoothing}, the prolongation operator defined in Section
\ref{sec:prolongation}, and the adjoint of the prolongation operator as
restriction operator. Then a W-cycle scheme with sufficiently many
smoothing steps leads to a solver with convergence independent of the
number of levels in the multigrid hierarchy and the parameter $\gamma$.
\end{theorem}
\section{Numerical example}\label{sec:numerics}
We conclude with a numerical example that corroborates the multigrid convergence theorem and that demonstrates the necessity of both components of the multigrid scheme.
The example is implemented using the Firedrake~\citep{Rathgeber:2016aa} finite element package, with the local solves performed using the PCPATCH~\citep{farrell2019pcpatch} preconditioner recently included in PETSc~\citep{petsc-user-ref}.

We consider the problem \eqref{eqn:elasticity} on a domain $\Omega = [0, 1]^d$ for a range of parameter
values $\gamma$.
We pick a zero right-hand-side $\fb=0$,
homogeneous Dirichlet conditions on the boundary $\{x_1=0\}$, a constant traction $\hb$ pulling down in $x_2$ direction with magnitude $1/2$ on the boundary $\{x_1=1\}$, and homogeneous Neumann conditions on the remaining boundaries.
The problem is discretised with Scott--Vogelius elements on Alfeld splits of a regular simplicial mesh.
We apply conjugate gradients preconditioned by multigrid W-cycles with
two Chebyshev smoothing iterations per level to solve the problem.
We consider four variants of the algorithm,
considering all combinations of robust and standard components. In
addition, we also compare to the HYPRE BoomerAMG
algorithm~\citep{henson2002} using symmetric SOR smoothing on each
process and additive updates between processes~\citep{baker2011}.
Each solver was terminated when the Euclidean norm of the
residual was reduced by eight orders of magnitude, up to a maximum
of 200 iterations.

Results for the two-dimensional problem with $[\Ptwo]^2$ are given in
Table \ref{tab:svgraddiv2d}, and results for the three-dimensional
problem with $[\Pthree]^3$ are given in Table \ref{tab:svgraddiv3d}.
The algebraic multigrid algorithm exhibits growth of iteration counts as
the mesh is refined.  For all four geometric multigrid variants, mesh
independence is observed, but parameter-robustness is only achieved when
both robust relaxation and robust transfer are employed.

\begin{table}[htbp]
\centering
\resizebox{0.99\textwidth}{!}{
\begin{tabular}{cr|cccccccc}
\toprule
\multirow{2}{*}{Refinements} & \multirow{2}{*}{DoF} & \multicolumn{8}{c}{$\gamma$} \\
                                && $0$	&	$1$	&	$10$	&	$10^2$ 	& $10^3$ &	$10^4$	&	$10^6$	&	$10^8$\\
\midrule
\multicolumn{10}{c}{Robust relaxation \& robust transfer}\\
\midrule
1 &   1\,602 & 9 & 9 & 10 & 14 & 15 & 15 & 15 & 15 \\
2 &   6\,274 & 9 & 9 & 11 & 15 & 15 & 15 & 15 & 15 \\
3 &  24\,834 & 9 & 9 & 11 & 15 & 15 & 15 & 15 & 15 \\
4 &  98\,818 & 8 & 9 & 10 & 15 & 15 & 15 & 15 & 15 \\
5 & 394\,242 & 8 & 8 & 10 & 14 & 15 & 15 & 15 & 15 \\
\midrule
\multicolumn{10}{c}{Robust relaxation \& standard transfer}\\
\midrule
1 &   1\,602 & 9 & 9 & 11 & 20 & 79   &  180 & >200 & >200 \\
2 &   6\,274 & 9 & 9 & 11 & 20 & >200 & >200 & >200 & >200 \\
3 &  24\,834 & 9 & 9 & 11 & 20 & >200 & >200 & >200 & >200 \\
4 &  98\,818 & 8 & 9 & 11 & 20 & >200 & >200 & >200 & >200 \\
5 & 394\,242 & 8 & 9 & 11 & 20 & >200 & >200 & >200 & >200 \\
\midrule
\multicolumn{10}{c}{Jacobi relaxation \& robust transfer}\\
\midrule
1 &   1\,602 & 21 & 23 & 51 & 173 & >200 & >200 & >200 & >200 \\
2 &   6\,274 & 21 & 24 & 53 & 180 & >200 & >200 & >200 & >200 \\
3 &  24\,834 & 21 & 23 & 53 & 181 & >200 & >200 & >200 & >200 \\
4 &  98\,818 & 21 & 23 & 52 & 179 & >200 & >200 & >200 & >200 \\
5 & 394\,242 & 20 & 23 & 52 & 179 & >200 & >200 & >200 & >200 \\
\midrule
\multicolumn{10}{c}{Jacobi relaxation \& standard transfer}\\
\midrule
1 &   1\,602 & 21 & 24 & 48 & >200 & >200 & >200 & >200 & >200 \\
2 &   6\,274 & 21 & 25 & 46 & >200 & >200 & >200 & >200 & >200 \\
3 &  24\,834 & 21 & 24 & 46 & >200 & >200 & >200 & >200 & >200 \\
4 &  98\,818 & 21 & 24 & 46 & >200 & >200 & >200 & >200 & >200 \\
5 & 394\,242 & 20 & 24 & 46 & >200 & >200 & >200 & >200 & >200 \\
\midrule
\multicolumn{10}{c}{Algebraic multigrid}\\
\midrule
1 &   1\,602 & 17 & 19 & 34 &  93 & 180  & >200 & >200 & >200 \\
2 &   6\,274 & 19 & 21 & 37 & 108 & >200 & >200 & >200 & >200 \\
3 &  24\,834 & 20 & 23 & 40 & 120 & >200 & >200 & >200 & >200 \\
4 &  98\,818 & 22 & 25 & 46 & 132 & >200 & >200 & >200 & >200 \\
5 & 394\,242 & 23 & 26 & 49 & 142 & >200 & >200 & >200 & >200 \\
\bottomrule
\end{tabular}
}
\caption{Iteration counts in two dimensions for the $[\Ptwo]^2$ element for five different geometric and algebraic multigrid variants.
The geometric multigrid results are obtained with a $4\times4$ coarse grid.}
\label{tab:svgraddiv2d}
\end{table}

\begin{table}[htbp]
\centering
\resizebox{0.99\textwidth}{!}{
\begin{tabular}{cr|cccccccc}
\toprule
\multirow{2}{*}{Refinements} & \multirow{2}{*}{DoF} & \multicolumn{8}{c}{$\gamma$} \\
                                && $0$	&	$1$	&	$10$	&	$10^2$ 	& $10^3$ &	$10^4$	&	$10^6$	&	$10^8$\\
\midrule
\multicolumn{10}{c}{Robust relaxation \& robust transfer}\\
\midrule
1 &     23\,871 & 17 & 15 & 14 & 18 & 23 & 25 & 26 & 25 \\
2 &    185\,115 & 18 & 16 & 16 & 21 & 26 & 29 & 31 & 30 \\
3 & 1\,458\,867 & 18 & 16 & 16 & 22 & 27 & 31 & 32 & 31 \\
\midrule
\multicolumn{10}{c}{Robust relaxation \& standard transfer} \\
\midrule
1 &     23\,871 & 17 & 20 & 38   &  124 & >200 & >200 & >200 & >200 \\
2 &    185\,115 & 18 & 23 & >200 & >200 & >200 & >200 & >200 & >200 \\
3 & 1\,458\,867 & 18 & 22 & >200 & >200 & >200 & >200 & >200 & >200 \\
\midrule
\multicolumn{10}{c}{Jacobi relaxation \& robust transfer}   \\
\midrule
1 &     23\,871 & 53 & 57 & 114 & >200 & >200 & >200 & >200 & >200 \\
2 &    185\,115 & 60 & 64 & 114 & >200 & >200 & >200 & >200 & >200 \\
3 & 1\,458\,867 & 58 & 64 & 114 & >200 & >200 & >200 & >200 & >200 \\
\midrule
\multicolumn{10}{c}{Jacobi relaxation \& standard transfer}\\
\midrule
1 &     23\,871 & 53 & 66 & 183  & >200 & >200 & >200 & >200 & >200 \\
2 &    185\,115 & 60 & 73 & >200 & >200 & >200 & >200 & >200 & >200 \\
3 & 1\,458\,867 & 58 & 73 & >200 & >200 & >200 & >200 & >200 & >200 \\
\midrule
\multicolumn{10}{c}{Algebraic multigrid}\\
\midrule
1 &     23\,871 & 42 & 45 & 80 & >200 & >200 & >200 & >200 & >200 \\
2 &    185\,115 & 46 & 52 & 90 & >200 & >200 & >200 & >200 & >200 \\
3 & 1\,458\,867 & 51 & 56 & 97 & >200 & >200 & >200 & >200 & >200 \\
\bottomrule
\end{tabular}
}
\caption{Iteration counts in three dimensions for the $[\Pthree]^3$ element for five different geometric and algebraic multigrid variants.
The geometric multigrid results are obtained with a $2\times2\times 2$ coarse grid.}
\label{tab:svgraddiv3d}
\end{table}

\review{To give an idea of relative computational costs, we report runtimes of the different solver components in two and three dimensions
in Table \ref{tab:timing}. To cleanly separate setup and application times, the solver was adjusted to use Richardson instead of Chebyshev
iteration for the relaxation; when configuring the Chebyshev relaxation, PETSc automatically estimates the eigenvalues of
the preconditioned operator, which requires applications of the subspace correction in the overall preconditioner setup. Both
relaxation
and prolongation are of $\mathcal{O}(1)$ complexity in $\gamma$; the cost of relaxation is linear in the number of vertices of the macro mesh on shape regular meshes,
while the cost of prolongation is linear in the number of macro cells.
}

\begin{table}[htbp]
\centering
\begin{tabular}{r|cc|cc|c}
\toprule
\multirow{2}{*}{Dimension} &\multicolumn{2}{c|}{Relaxation}&\multicolumn{2}{c|}{Transfer} & \multirow{2}{*}{Other}\\
& Setup & Application & Setup & Application &\\
\midrule
2 & 19.0 & 15.5 & 6.8 & 7.5 & 4.1\\
3 & 83.6 & 77.0 & 20.7 & 17.0 & 12.8\\
\bottomrule
\end{tabular}
\caption{\review{Runtime (in seconds) of the different solver components for simple two level examples in two and three dimensions.
To simplify timing the code is run in serial and we employ Richardson
iteration on each level. In 2D the fine grid has 221954 dofs, and in 3D
the fine grid has 185115 dofs.}}
\label{tab:timing}
\end{table}

\section{Conclusion}
We have demonstrated that it is possible to construct multigrid methods in both
two and three dimensions that are robust with respect to the parameter $\gamma$
in the nearly incompressible elasticity equations~\eqref{eqn:elasticity}.
This includes situations where the meshes are nonnested, for which there is a
need to impose incompressibility constraints as part of the prolongation
operator.
Our main contribution is to guide the choice of subspace decomposition used by
the smoother without the need for any explicit description of the
divergence-free space, by utilizing local Fortin operators guaranteed by local
inf-sup conditions.
This approach can likely be applied more generally when a local inf-sup condition is known.

\section*{Acknowledgements}
This research is supported by the Engineering and Physical Sciences
Research Council [grant number EP/R029423/1 and EP/V001493/1], by the EPSRC Centre For Doctoral Training in
Industrially Focused Mathematical Modelling [grant number EP/L015803/1] in
collaboration with London Computational Solutions. LM also acknowledges
support from the UK Fluids Network [EPSRC grant number EP/N032861/1] for
funding a visit to Oxford. FW was partially supported by a grant from the Simons Foundation (560651)

\section*{Code availability}
For reproducibility, we cite archives of the exact software versions used to
produce the results in this paper. All major Firedrake components as well as
the code used to obtain the shown iteration counts have been
archived on Zenodo \citep{zenodo/Firedrake-20210702.0}.
An installation of Firedrake with
components matching those used to produce the results in this paper can by
obtained following the instructions at
\url{https://www.firedrakeproject.org/download.html}.

\ifarxiv
\bibliography{fortin-subspace-decomposition.bbl}
\else
\bibliographystyle{IMANUM-BIB}
\bibliography{latex-includes/TEX-Bib}
\fi
\end{document}